\documentclass[11pt]{article}   
\usepackage{amsfonts,amsmath,latexsym}
\usepackage{bm}
\usepackage[T1]{fontenc}
\usepackage[utf8]{inputenc}
\usepackage{verbatim}
\usepackage{algorithm}
\usepackage{algpseudocode}
\usepackage{eurosym}
\usepackage{enumitem}
\usepackage{dsfont}
\usepackage{hyperref}
\usepackage[dvips]{graphicx}
\usepackage{epsfig}
\usepackage{epsf}
\usepackage{mathrsfs}
\usepackage{amsthm}
\usepackage{color}
\usepackage{amssymb}
\usepackage{comment}

\usepackage{caption}
\usepackage{subfigure}

 \usepackage{fancyhdr} 

\usepackage{palatino}
\def\R{{\mathbb R}}
\def\E{{\mathbb E}}
\def\P{{\mathbb P}}
\def\N{{\mathbb N}}
\def\F{{\mathcal F}}
\def\B{{\mathcal B}}
\def\Pma{{\mathcal P}}
\def\shd{{\mathcal D}}

\def\1{{\mathds 1}}
\def\vphi{{\varphi}}
\newtheorem{theorem}{Theorem}[section]
\newtheorem{prop}[theorem]{Proposition}
\newtheorem{coro}[theorem]{Corollary}
\newtheorem{remark}[theorem]{Remark}
\newtheorem{example}[theorem]{Example}

\newtheorem{lemma}[theorem]{Lemma}
\newtheorem{hyp}[theorem]{Hypothesis}

\newtheorem{definition}[theorem]{Definition}
\newtheorem{property}[theorem]{Property}
\newtheorem{notation}[theorem]{Notation}

\renewcommand{\theequation}{\arabic{section}.\arabic{equation}}

\DeclareMathAlphabet{\mathonebb}{U}{bbold}{m}{n}
\newcommand{\one}{\ensuremath{\mathonebb{1}}}

\def \D{\mathbb D}

\def\R{\mathbb R}
\def\N{\mathbb N}

\def\E{\mathbb E}
\def\P{\mathbb P}
\def\Q{\mathbb Q}

\def\sha{{\cal A}}
\def\shb{{\cal B}}
\def\shc{{\cal C}}
\def\shd{{\cal D}}

\def\shf{{\cal F}}
\def\shg{{\cal G}}
\def\shh{{\cal H}}

\def\shm{{\cal M}}
\def\shn{{\cal N}}
\def\shl{{\cal L}}
\def\shp{{\cal P}}
\def\shs{{\cal S}}

\def\scrd{{\mathscr D}}

\author{
  {\sc Thibaut BOURDAIS}
	\thanks{ENSTA Paris, Institut Polytechnique de Paris.
		Unit\'e de Math\'ematiques Appliqu\'ees (UMA).
		E-mail:{ \tt thibaut.bourdais@ensta.fr}} 
	{\sc,}\ {\sc Nadia OUDJANE}
	\thanks{EDF R\&D,   and FiME (Laboratoire de Finance des March\'es de l'Energie
		(Dauphine, CREST,  EDF R\&D) www.fime-lab.org). 
		E-mail:{\tt  
			nadia.oudjane@edf.fr}}
	\ {\sc and}\ {\sc Francesco RUSSO} 
	\thanks{ENSTA Paris, Institut Polytechnique de Paris.
		Unit\'e de Math\'ematiques Appliqu\'ees (UMA). 
		E-mail:{\tt  francesco.russo@ensta.fr}.
}}

\date{August 2nd 2026}

\title{Exponential twist of probability measures: drift correction in term of a generalized gradient.}

\newcommand{\MBFigure}[6]{
	$\left. \right.$ \\
	\refstepcounter{figure}
	\addcontentsline{lof}{figure}{\numberline{\thefigure}{\ignorespaces #5}}
	\begin{center}
		\begin{minipage}{#1cm}
			\centerline{\includegraphics[width=#2cm,angle=#3]{#4}}
			\begin{center}
				\upshape{F\textsc{ig} \normal
				\end{center}
				size{\thefigure}. $-$} #5
		\end{center}
		\label{#6}
	\end{minipage}
\end{center}
$\left. \right.$ \\}

\begin{document}
\maketitle

\begin{abstract}
  In this paper we study the exponential twist, i.e.
 a path-integral exponential change of measure,
 of a Markovian reference probability measure $\P$.
 This type of transformation naturally appears in variational representation formulae originating from the theory of
 large deviations and can be interpreted in some cases, as the solution of a specific stochastic control problem.
 Under a very general Markovian assumption on $\P$, we fully characterize the exponential twist probability measure as the solution of a martingale problem
 and prove that it inherits the Markov property of the reference measure.
 The ''generator'' of the martingale problem shows
 a drift depending on a {\it generalized gradient} of some
 suitable {\it value function} $v$.
 The analysis focuses on the fact that any Markovian probability measure
 fulfills an {\it intrinsic martingale problem} for which no
 uniqueness is required.
 \end{abstract}
 
\medskip\noindent {\bf Key words and phrases:}  
Stochastic control;  optimization;
exponential twist; generalized gradient; relative entropy.

\medskip\noindent  {\bf 2020  AMS-classification}: 60H10; 
60H30; 60J60; 65C05; 49L25;
35K58.

\medskip

\section{Introduction}


This paper focuses on exponential twist probability measures $\Q$ resulting from an exponential change of measure with respect to a Markovian reference probability measure $\P$, i.e.
\begin{equation}
	\label{eq:introOptiSolution}
	d\Q \propto e^{-\vphi} d\P,
\end{equation}
when $\vphi$ is a \textit{path-integral} functional of the form
\begin{equation} \label{eq:decPhig}
	\vphi(X) = \int_0^T f(r, X_r)dr +g(X_T),
\end{equation}
for measurable functions $f, g$.
When the reference probability  $\P$ is the Wiener measure,
 the properties of $\Q$ have been extensively studied in
 \cite{EntropyWeighted}. 
 In the case of discrete time Markov models with finite state space, the stability of the Markov property by the exponential twist transformation \eqref{eq:introOptiSolution} was already pointed out in \cite{KLQ, KLQPowerDemand, ExpoTwistGeneral}. In this paper we extend these results
to all Markovian models, including
 c\`adl\`ag (possibly singular) Markovian SDEs and provide a precise characterization of the generator associated to the martingale problem verified by the probability measure $\Q$.

  Exponential twist probability measures of the form \eqref{eq:introOptiSolution} are intimately connected to various application domains.
It appears naturally in variational representation formulae in relation to the theory of large deviations \cite{SanovLargeDeviations, DupuisEllisLargeDeviations}.
In fact we have the variational formula 
 \begin{equation}
 	\label{eq:variationalFormulation}
 	- \log \int_\Omega e^{- \vphi(\omega)}d\P(\omega) = \inf_{\Q \in \Pma(\Omega)} \int_\Omega \vphi(\omega)d\Q(\omega) + H(\Q|\P),
 \end{equation}
  see Proposition 1.4.2 in \cite{DupuisEllisLargeDeviations},
 where
 $\Pma(\Omega)$ is the space of all probability measures
 on $(\Omega, \F)$ and $H$ is the relative entropy of $\Q$
 with respect to $\P$, see Definition
\ref{def:klDiv}. The minimum in \eqref{eq:variationalFormulation} is achieved for the exponential twist probability measure $d\Q \propto e^{-\vphi}d\P$ and 
$\Q$ is said to be the \textit{solution} to the optimization problem \eqref{eq:variationalFormulation}.

In fact, the characterization results for the exponential twist measure \eqref{eq:introOptiSolution} provided in the present paper
apply to the framework of non-linear optimization on the space of probability measures (often related to mean-field optimization \cite{MFOptimization})
 stated as
\begin{equation}
	\label{eq:generalizedCharacterization}
	\inf_{\Q \in \shp(\Omega)} F\left(\E^\Q[\vphi(X)]\right) + H(\Q | \P),
\end{equation}
where $F : \R \rightarrow \R$ is a differentiable convex function and
$\vphi$ is given by \eqref{eq:decPhig}.
Indeed, assume that Problem \eqref{eq:generalizedCharacterization} admits a solution $\tilde \Q$. Then $\tilde \Q$ is also solution of
\begin{equation}
	\label{eq:linearizedProblem}
	\inf_{\Q \in \shp(\Omega)} \E^\Q[\tilde\vphi(X)] + H(\Q | \P),
\end{equation}
where $\tilde\vphi(X) := F'\left(\E^{\tilde \Q}[\vphi(X)]\right)\vphi(X)$ (see Lemma \ref{lemma:equivalenceOpti})
and $\Q$ is an exponential twist measure of the form \eqref{eq:introOptiSolution} with $\vphi$ replaced by $ \tilde \vphi$.
Hence, any solution of Problem \eqref{eq:generalizedCharacterization} falls into the framework of the present paper.
Optimization programs of the form \eqref{eq:generalizedCharacterization} appear for example in \cite{KLQ, KLQPowerDemand, seguret, seguret2} for demand side management in power systems.

 The first crucial observation in this paper is that the only assumption on $\P$ to be
 Markovian
determines a natural domain $\shd(\P)$ of an (intrinsic) Markovian martingale
problem, see Definition \ref{def:domainP}.
This also identifies a map $a=a^\P: \shd(\P) \rightarrow L^0(\P)$
(the {\it generator}),
where $L^0(\P) = \{\phi:[0,T] \times \R^d \rightarrow \R \vert \phi \in
L^0(dt \otimes d\P_{X_t})$,
 is the natural basic space for our developments.
In particular, for every $\phi \in \shd(\P),$
$\phi(t,X_t) = \phi(0,X_0) + \int_0^t a^\P\phi(r,X_r) dr + M[\phi]$,
where $M[\phi]$ is a locally square integrable local martingale.
 When $\P$ is a solution  of
 a càdlàg Markovian martingale problem of domain $\shd$
 with respect to a (generally PDE or PIDE) map $\shl$
 (see Definition \ref{def:martingaleProblem}),
necessarily $\shd \subset \shd(\P)$ and
the restriction of $\shl$ is a restriction of $\shd(\P)$
to $\shd$. This includes the case when $\P$ is the law of a
stochastic differential equation (SDE) with jumps, in the case of
singular (distributional) coefficients.
The notion of martingale problem was introduced by
 D.W. Stroock and S.R.S Varadhan in the seminal papers
 \cite{sv1_1969, sv2_1969} and has been exploited extensively
 starting from \cite{stroock}. In that framework $\shd$
 was given, for instance by $C^{1,2}$-functions, possibly bounded
 or with compact support.

 In our approach, there is no issue of well-posedness for the considered
 martingale
problems. 
 Given a general
 Markovian probability $\P$, the probability $\Q$,
 defined in \eqref{eq:introOptiSolution}, 
 is still Markovian by Proposition \ref{lemma:QMarkovProp}.
The  main objective of the paper is then to 
 characterize       
 the Markovian martingale problem,
in particular to determine the generator $a^\Q$ defined
in the domain $\shd(\Q)$ in a precise way, exhibiting a class
of important examples.
In Section \ref{sec:characterization},
we  identify the so called {\it Ideal Property} \ref{cond:idealCondition}
associated with the reference probability measure $\P$ and a functional domain $\shd$, in relation of what we call the {\it intrinsic value function} $v$
introduced in Definition \ref{definitionV} as a class of functions
defined on $[0,T]\times \R^d$.
 One difficulty in the paper comes from the fact that, in general, the natural domain
 $\shd(\P)$ is not an algebra.
 For this reason, the role of  Theorem 
 \ref{thm:IdealCond} is relevant: it shows that $\P$ verifies
 the Ideal Property
 with respect to every subalgebra $\shd$ of $\shd(\P)$.
 Suppose that a martingale problem is fulfilled for $(\P, \shd)$
and that the Ideal Property with respect to $\shd$ is fulfilled.
Theorem \ref{th:markovDrift} 
states that 
$ \shd \subset \shd(\P) \subset \shd(\Q)$ and
the ''generator'' $a^{\Q}$ of the corresponding Markovian martingale problem 
is explicitly expressed as follows.
For all test function $\phi \in \shd$, $a^{\Q}(\phi) = a^\P(\phi) + \Gamma^v(\phi)/v$,
and $\Gamma^v(\phi)/v$ is a correction term identified via a Girsanov's change of measure, associated with the intrinsic value function $v$.

In Section \ref{sec:extension}, we further specify the map
$ \shd \ni \phi \mapsto \Gamma^v(\phi)$ in the integro-differential
case. In particular Proposition \ref{prop:GammaContinuous} states that it can be expressed
as the sum of an integral term corresponding to the jumps contribution, and a
\textit{generalized gradient}  $\Gamma^{v,c}$ of $v$.
On the other hand Corollary \ref{cor:P1PDE} shows that if there is a solution $w \in C^{0,1}$
of $a^\P w = f w$, $w$ is a version of the intrinsic value function
and the generalized gradient can be expressed as $\Gamma^{v,c}(\phi) =
(\nabla_x \phi)^\top \sigma \sigma^\top\nabla_x w$,
independently on the fact that the canonical process is a semimartingale.
Then, in Proposition \ref{prop:extensionGammaClosure}, we extend the ``out of jumps component''
$\Gamma^{v,c}$ of the operator $\Gamma^{v}$
to a larger space $\scrd$ including the identity function $id$ and other test functions
outside the domain $\shd$. This extension allows
 us, in Proposition \ref{prop:idGamma},  to express the change of probability measure as a drift
 modification depending only on $\Gamma^{v,c}(id)$.
In particular,
even when the initial drift is a Schwartz distribution,
an additional measurable drift term appears, extending the notional term 
$ \sigma \sigma^\top\nabla_x v$ when $\nabla_x v$ does not exist.

Section \ref{sec:checkingHypV} 
is devoted to the verification of Ideal Property as soon as $\shd$ is an algebra: this
is done via delicate ``Dellacherie-Meyer'' type arguments.

In Section \ref{sec:applications}, we instantiate our characterization result (Theorem \ref{th:markovDrift}) on several specific examples.
We first study the case where the reference probability measure $\P$ is solution to a martingale problem associated to a jump diffusion. In this situation we are able to fully characterize in Proposition \ref{prop:markovianDriftJumps} the drift of the canonical process under the optimal probability $\Q$, as well as its jump intensity,
as Markovian functions of the current state. We emphasize  that we do not require any integrability condition of the underlying process with respect to $\P$. We then apply these results to the case of (even very singular) Brownian diffusions
deriving in Corollary \ref{coro:markovianDrift}   
the  drift correction related to $\Q$,  as a Markovian function
and we characterize it by means of a generalized gradient.
We finally consider more irregular examples in Proposition \ref{prop:markovianDriftDistri}, where the drift $b$ related to reference probability $\P$, is even a Schwartz distribution.

In a companion paper \cite{BOROptimi2023} and in its complete version \cite{BOROptimi2023Complete},
we make use of 
 Corollary \ref{coro:markovianDriftSDE}
 which is a consequence of  Corollary \ref{coro:markovianDrift}.
This is a basic tool
which allows us to develop an algorithm that provides
Markovian controls  approximating
the solution to a large class
of stochastic control problems.
In that framework no particular hypotheses (in particular no integrability) are required for the cost functions $f$ and $g$, beyond measurability
and a lower bound.


The paper is organized as follows. After a preliminary section of notations, in Section \ref{sec:characterization},
we characterize the general Markovian martingale problem associated with the given reference probability,
and we provide some suitable calculus with respect to a related generator and ''carré du champs operator''.
We also characterize in full generality the Markovian martingale problem verified by the
exponential twist probability. We also introduce the {\it Ideal Property} associated with
a linear subspace $\shd$ of the general Markov domain. 
In Section \ref{sec:extension}, we identify a continuous component of the carré du champs operator
and we provide an extension in full generality.
In Section \ref{sec:checkingHypV} we show that the Ideal Property is always verified
with respect to any Markovian probability and each subalgebra  $\shd$  of the
Markovian domain.

Section \ref{sec:applications} contains various applications to general jump-diffusion processes
and solutions to SDE with distributional drift.

\section{Notations and definitions}
\label{sec:notation}

\setcounter{equation}{0}

In this section we introduce the basic notions and notations used throughout this document. In what follows, $T  > 0$ will be a fixed time horizon.
\begin{itemize}
	\item All vectors $x \in \R^d$ are column vectors. Given $x \in \R^d$, $|x|$ will denote its Euclidean norm.
	
	\item Given a matrix $A \in \R^{d \times d}$, $\|A\| := \sqrt {Tr[AA^\top]}$ will denote its Frobenius norm.
	
	\item For any $x \in \R^d$, $\delta_x$ will denote the Dirac mass in $x$.
	
	\item $\shc^{i, j} := C^{i, j}([0, T] \times \R^d, \R)$ will be the
          $F$-space of real valued functions on $[0, T] \times\R^d$ that are continuous together with their time and space derivatives up to order $i$ and $j$ respectively. It is endowed with the topology of uniform convergence on compact sets.
	
	\item $\shc^{i, j}_b := C_b^{i, j}([0, T] \times \R^d, \R)$ will be
          the Banach space of functions belonging to $\shc^{i, j}$ which
          are bounded  together with their time and space derivatives up to order $i$ and $j$ respectively. It is endowed with the topology of uniform convergence.

        \item Given a metric space $E$ we will denote by
          ${\rm Bor}(E)$ its Borel $\sigma$-field.
          
	\item For any  metric spaces $E$ and $F,$
          $C(E, F)$ (resp. $C_b(E, F)$, $\B(E, F)$, $\B_b(E, F)$)
          will denote the linear space of functions from $E$ to $F$ that are continuous (resp. bounded continuous, Borel, Borel bounded).
          If $E = F$ we will simply denote $C(E)$ (resp. $C_b(E)$, $\B(E)$, $\shb_b(E)$) for $C(E, E)$ (resp. $C_b(E, E)$, $\B(E, E)$, $\B_b(E, E)$).
          $\Pma(E)$ will denote the set of Borel probability measures on $E$. Given $\P \in \Pma(E),$ $\E^\P$ will denote the expectation with respect to (w.r.t.) $\P.$
 	
	\item Given $\phi \in C^{1, 2}([0, T] \times \R^d, \R)$, $\partial_t\phi$, $\nabla_x \phi$ and $\nabla_x^2\phi$ will denote respectively the partial derivative of $\phi$ with respect to $t \in [0, T]$, its gradient and its Hessian matrix  w.r.t. $x \in \R^d$.
	Given any bounded function $\phi$ we
	will denote by $\Vert \phi \Vert_\infty$ its supremum.
	
	\item For $x \in \R^d$, $id(x) := (id_i(x))_{1 \le i \le d} := (x_i)_{1 \le i \le d}$ will denote the identity on $\R^d$.
	
	\item 
	Given $0 \le t \le T$, $D([t, T], \R^d)$
	will denote the space  of c\`adl\`ag functions defined on $[t, T]$
	with values in $\R^d$.
	In the whole paper   $\Omega$ will denote the space
	$D([0,T])$.
	For any $t \in [0, T] $ we denote by $X_t : \omega \in \Omega \mapsto \omega_t$ the coordinate mapping on $\Omega.$ We introduce the $\sigma$-field $\F := \sigma(X_r, 0 \le r \le T)$. On the measurable space $(\Omega, \F),$ we introduce the \textit{canonical process} $X : (t,\omega) \in ([0, T] \times \Omega, \B([0, T])\otimes \F) \mapsto X_t(\omega) = \omega_t \in (\R^d, {\rm Bor}(\R^d))$.
	
	We endow $(\Omega, \F)$ with the right-continuous filtration $\F_t := \underset{t < s \le T}{\bigcap}\sigma(X_r, 0 \le r \le s).$ The filtered space
	$(\Omega, \F, (\F_t))$ will be called the \textit{canonical space}; for the sake of brevity, we denote $(\F_t)_{t \in [0, T]}$ by $(\F_t)$.
	
	For $0 \le t \le u \le T$, we denote
  $\shf_{t, u} := \sigma\left(X_r, t \le r \le u\right)$ and, if $u < T$, \\
   $\shf_{t, u+} := \underset{n > 0}{\bigcap}\sigma\left(X_r, t \le r \le u + \frac{1}{n}\right)$.
   \item     $ \shp(\Omega)$ will denote the space of all the probabilities
        on $(\Omega, \shf)$. Since $\Omega$ is a separable Banach space,
        $\shf$ coincides with ${\rm Bor}(\Omega)$.
        
	\item Given $\P \in \shp(\Omega)$ and a generic $\sigma$-field $\shg$ on $\Omega$, $\shg^\P$ will denote the $\P$-completion of $\shg$.
	
	\item Given $\P \in \shp(\Omega)$, $\D^{ucp}(\P)$ will
          denote the space of all c\`adl\`ag adapted processes (indexed by
          $[0,T]$) endowed with the
     uniform convergence in probability (u.c.p.) topology under $\P$.
    \item 
      A process $(X_t)$ will be said locally square integrable
      if there is an increasing sequence of stopping times $(\tau_n)$
      diverging to $+\infty$ such that
      $\sup_t \vert X_{\tau_n \wedge t}\vert$ is square integrable.
      Given $\P \in \shp(\Omega)$, $\mathcal{H}^2_{loc}(\P)$ will denote the space of locally square-integrable martingales. Given $M, N \in \mathcal{H}^2_{loc}(\P),$
  $\langle M, N \rangle$ will denote their predictable \textit{(angle) bracket}.
  If $M = N,$ we will use the notation
  $\langle M\rangle.$

  We also denote $Pos(\langle M, N \rangle) := \frac{1}{4}\langle M + N \rangle$ and $Neg(\langle M, N \rangle) := \frac{1}{4}\langle M - N \rangle$.
  
	\item Given $\P \in \shp(\Omega)$, $\sha_{loc}(\P)$ will denote the set of c\`adl\`ag processes with $\P$-locally integrable variation.
	 
	\item Equality between stochastic processes are intended
          in the sense of \textit{indistinguishability}.
	
	\item Throughout the paper we will use the notion of random measures
          and their associated {\it compensator}. For a detailed discussion on this topic as well as some unexplained notations we refer to Chapter II and Chapter III in \cite{JacodShiryaev}.
          In particular, the \textit{compensator}
of a random measure is introduced in
 Theorem 1.8, Chapter II, 
\cite{JacodShiryaev}. We also make use of the compensator
of bounded variation process, which is introduced in Theorem 3.18, Chapter I
of \cite{JacodShiryaev}.

\item We will work with the convention that $\inf \emptyset = + \infty$. In particular, any hitting time $\tau$ of a Borel set by a stochastic process defined on $[0, T]$ will have values in $[0, T] \cup \{+ \infty\}$.
	
\end{itemize}

\begin{definition} (Relative entropy).
  \label{def:klDiv}
  Let $\P, \Q \in \Pma(\Omega).$ The \textit{relative entropy} $H(\Q | \P)$ between the measures $\P$ and $\Q$ is defined by
	\begin{equation}
		\label{eq:relativeEntropy}
		H(\Q | \P) :=
		\left\{
		\begin{aligned}
			& \E^{\Q}\left[\log \frac{d\Q}{d\P}\right] &\text{if $\Q \ll \P$}\\
			& + \infty &\text{otherwise,}
		\end{aligned}
		\right.
	\end{equation}
	with the convention $\log(0/0) = 0$.
\end{definition}
\begin{remark}
	\label{rmk:relativeEntropy}
	The relative entropy $H$ fulfills the following properties 
	for which we refer to \cite{DupuisEllisLargeDeviations} Lemma 1.4.3.
	\begin{enumerate}
		\item
		$H$ is non negative and jointly convex, that is for all $\P_1, \P_2, \Q_1, \Q_2 \in \Pma(\Omega)$, for all $\lambda \in [0, 1]$, $H(\lambda \Q_1 + (1 - \lambda) \Q_2 | \lambda \P_1 + (1 - \lambda)\P_2) \le \lambda H(\Q_1 | \P_1) + (1 - \lambda)H(\Q_2| \P_2)$.
		\item $(\P, \Q) \mapsto H(\Q | \P)$ is lower semicontinuous with respect to the weak convergence on Polish spaces.
	\end{enumerate}
\end{remark}
We introduce here a significant space of
      (equivalence classes)
      of Borel functions, associated with a given probability $\P$.

      \begin{notation} \label{rmk:v}
		\begin{equation}
			\label{eq:defL0}
			L^0 := L^0(\P) = \left\{\phi \in \shb([0, T] \times \R^d, \R)~:~\int_0^T |\phi(r, X_r)|dr < + \infty \quad
                          \P\text{-a.s.}\right\},		\end{equation}
		which corresponds to the classical space $L^0([0,T] \times \R^d, dt \otimes d\P_{X_t}),$
                where $\P_{X_t}$ is the (marginal) law of $X_t$ under $\P$.
                With a slight abuse of notations $L^0$ can be seen as a linear space of equivalence classes, where the equivalence is given
                by the equality up to  a  $dt \otimes d\P_{X_t}$ null set.
              \end{notation}
              \begin{remark} \label{rmk:L^0}
                We have $L^0(\P) = L^0(\Q)$ if
                $P$ and $Q$ are equivalent probabilities.
                \end{remark}
              \begin{definition}
	(Lévy kernel).
	\label{def:levyKernel}
	$L : [0, T] \times \R^d \times \shb(\R^d)$ is called a
        {\it (deterministic) Lévy kernel} if it satisfies the following.
	\begin{enumerate}
        \item For all $t \in [0,T], x  \in \R^d$,  $L(t, x, .)$ is a non-negative Borel measure on $\R^d$ such that $L(t, x, \{0\}) = 0$,
which  is  $\sigma$-finite  on $\R^d \backslash \{0\}$;
         	\item $(t, x) \mapsto \int_A (1 \wedge |q|^2)L(t, x, dq)$ is Borel  and bounded for all $A \in {\rm Bor}(\R^d)$.
                  	\end{enumerate}
                      \end{definition}
                      In the sequel we will often postulate the following hypothesis.
                      \begin{hyp}
	\label{hyp:compensator}
	(Compensator).
	The $\P$-compensator $\nu^{X, \P}$ of the jump measure $\mu^X$ of $X$ satisfies $\nu^{X, \P}(X, dt, dq) = dtL(t, X_{t-}, dq),$ where $L$ is a deterministic
        Lévy kernel in the sense of Definition \ref{def:levyKernel}.

        \end{hyp}
                      
\section{Characterization of the exponential twist measure}
\label{sec:characterization}
\setcounter{equation}{0}

\subsection{Martingale problem and Markov domain}

In the paper we will often use the notion of martingale problem.
Given a measurable process $Y$ we say that it admits
a càdlàg modification (with respect to $\P$)
if there is a càdlàg measurable process
$\tilde Y$ such that $Y_t(\omega) = \tilde Y_t(\omega)$
for almost all $t$, for $\P$-almost all $\omega$.
\begin{definition}
	\label{def:martingaleProblem}
	(Martingale problem).
	Let $a : \shd \subset \B([0, T] \times \R^d, \R) \rightarrow \B([0, T] \times \R^d, \R)$ be a linear operator.
	Let $\mu \in \Pma(\R^d)$. We say that a probability measure $\P \in \Pma(\Omega)$ is solution of the \textit{martingale problem associated to
		$(\shd, a, \mu)$} if
	
	\begin{enumerate}
		\item $\mathcal{L}^\P(X_0) = \mu$;
		\item for every $\phi \in \shd$ the process
		\begin{equation}
			\label{eq:MphiDef}
			M[\phi] := \phi(\cdot, X_\cdot) - \phi(0, X_0) - \int_0^\cdot a(\phi)(r, X_r)dr
                      \end{equation}
                      has a c\`adl\`ag modification which is a
       (locally square integrable) local martingale under $\P$.

	\end{enumerate}
\end{definition}

We will moreover assume that the reference probability measure $\P$ 
 has the Markov property
 below.
\begin{hyp}
	\label{hyp:markovProp}
	$\P$ satisfies the Markov property
	\begin{equation}
		\label{eq:markovPropDef}
		\E^{\P}[F((X_u)_{u \in [t, T]})|\F_t] = \E^{\P}[F((X_u)_{u \in [t, T]})|X_t],
	\end{equation}
	for all $F \in \B_b(D[t, T], \R)$, $t \in [0, T]$.
      \end{hyp}
      We introduce below the notion of Markov domain and generators
      directly related
      to the Markov property.

\begin{definition}
	\label{def:domainP}
	(Markov domain $\shd(\P)$ and associated generator  $a^\P$).
        \begin{itemize}
        \item A Borel
          function $\phi : [0, T] \times \R^d \rightarrow \R$ is an element of the Markov domain $\shd(\P)$ if
          there exists a Borel
        function $\chi \in \shb([0, T] \times \R^d, \R)$ such that the process
	\begin{equation}
		\label{eq:defMPhiP}
		M[\phi] := \phi(\cdot,X ) - \phi(0, X_0) - \int_0^\cdot \chi(r, X_r)dr,
              \end{equation}
	has a c\`adl\`ag modification in $\shh_{loc}^2(\P)$.
        That modification will still  be denoted $M[\phi]$.
\item
        In this way also $(\phi(\cdot,X))$ admits a càdlàg modification,
        which (when there is no ambiguity) will still be denoted by
        $(\phi(\cdot,X_\cdot))$.
   \item      We will also denote
        $a^\P: \chi \mapsto a^\P(\phi) := \chi$.
\end{itemize}
      \end{definition}

 From now on we will make use of the linear space $L^0:= L^0(\P)$ defined in Notation \ref{eq:defL0}.
                \begin{remark} \label{rmk:MP}
	\begin{enumerate}
        \item $a^\P(\phi)$ defines a $dt \otimes d\P_{X_t}$-unique element of
$L^0$.
Indeed assume that there exist two elements $\chi_1$ and $\chi_2$ of $\shb([0, T] \times \R^d, \R)$ such that \eqref{eq:defMPhiP} holds for $\chi = \chi_1$ or $\chi_2$. Clearly $\phi(\cdot, X_\cdot)$ is a special semimartingale under $\P$ and uniqueness of the decomposition of special semimartingales immediately yields $\int_0^t \chi_1(r, X_r)dr = \int_0^t \chi_2(r, X_r)dr$ $\P$-a.s. for all $t \in [0, T]$, that is $\chi_1 = \chi_2$ $dt \otimes d\P_{X_t}$-a.e.
\item Also if $\phi^1 = \phi^2$ in $L^0$ then $M[\phi^1] = M[\phi^2]$
  up to indistinguishability.
\item If $\phi \in \shd(\P)$ we will denote by
  $(\phi_{-}(t,X_t))$ the (càglad) process
  $\lim_{s \rightarrow t-} \phi(s,X_s)$.
  
\item $\P$ is (naturally) the solution of the martingale problem associated with $(\shd(\P), a^\P)$.
\item $\shd(\P)$ is a linear space.
  If $\lambda, \mu \in \R$ and $\phi^1, \phi^2 \in \shd(\P)$
  then
  $M[\lambda \phi^1 + \mu \phi^2] =  \lambda M[\phi^1] + \mu M[\phi^2]$.
                \item Assume that $\P$ is solution of a martingale problem associated to $(\shd, a)$ in the sense of Definition \ref{def:martingaleProblem}.
 If $\P$ fulfills the martingale problem
  with respect to $(\shd, a)$, then
 $\shd  \subset \shd(\P)$ and $a$ is a restriction
 of $a^\P$ to $\shd$.

\item 	When $\shd \subset \shd(\P)$ is a linear subspace of $\shd(\P)$ which is also an algebra ({\it i.e.} stable by multiplication),
  $\shd$ will be called a \textbf{subalgebra} of $\shd(\P)$.
\end{enumerate}
\end{remark}

      \subsection{The exponential twist measure and its Markov property}
        
We consider now $f, g$ verifying Hypothesis \ref{hyp:costFunctions} below.
\begin{hyp}
	\label{hyp:costFunctions}
	(Cost functions). $f \in \B([0, T] \times \R^d, \R)$, $g \in \B(\R^d, \R)$ and $f, g \ge 0$ such that
       $f \in L^0(\P)$.
      \end{hyp}
Hypothesis \ref{hyp:costFunctions} will also be in force in the whole paper.
             \begin{remark} \label{rmk:costFunctions}
Without restriction of generality, Hypothesis \ref{hyp:costFunctions}
can be relaxed supposing $f, g$ to be lower bounded.
\end{remark}

As mentioned earlier,
in this paper we aim at characterizing  the exponential twist $\Q$ defined by
\begin{equation} \label{eq:Qstar}
  d\Q := D_Td\P,
  \end{equation}
  where
  
\begin{equation} \label{eq:optimalDensity}
  D_T := \frac{\exp\left(-\int_0^T f(r, X_r)dr - g(X_T)\right)}{\E^{\P}\left[\exp\left(-\int_0^T f(r, X_r)dr - g(X_T)\right)\right]}.
  \end{equation}

A first important observation concerns the fact that the exponential twist $\Q$
conserves the Markov property, i.e. $\Q$
still fulfills Hypothesis \ref{hyp:markovProp}.
  \begin{prop}
       \label{lemma:QMarkovProp} 
        Let $\P$ our reference probability  supposed to fulfill
        Markov property Hypothesis \ref{hyp:markovProp}. 
        In this lemma we do not necessarily suppose
        that $g \ge 0$ but only that
       $ \exp\left(- g(X_T)\right) \in L^1(\P)$.
        
        Then, the probability $\Q$  defined by (\ref{eq:Qstar}),
 also verifies the same Markov property.
\end{prop}
\begin{proof}
	Let $t \in [0, T]$ and $F \in \B_b(D([t, T], \R^d), \R)$. It holds
	\begin{equation}
		\label{eq:formulaCondExp}
		\E^{\Q}\left[F\left(\left(X_r\right)_{r \in [t, T]}\right)\middle | \F_t\right] = \frac{\E^{\P}\left[F\left(\left(X_r\right)_{r \in [t, T]}\right)\frac{d\Q}{d\P}\middle | \F_t\right]}{\E^{\P}\left[\frac{d\Q}{d\P}\middle | \F_t\right]}.
	\end{equation}
       Setting ${\rm d}: = \frac{1}{\E^{\P}\left[\exp\left(-\int_0^T f(r, X_r)dr - g(X_T)\right)\right]}$,
by the validity of the Markov property  for $\P$, we have
	\begin{equation}
		\label{eq:condExp2}
		\begin{aligned}
                  &\E^{\P}\left[F\left(\left(X_r\right)_{r \in [t, T]}\right)\frac{d\Q}{d\P}\middle | \F_t\right]\\
                  =&{\rm d} \exp\left(-\int_0^t f(r, X_r)dr\right) \E^{\P}\left[F\left(\left(X_r\right)_{r \in [t, T]}\right)
                     \exp\left(- \int_t^T f(r, X_r)dr - g(X_T)\right)\middle | X_t\right].
		\end{aligned}
              \end{equation}
              Consequently, taking $F \equiv 1$, we get
\begin{equation}
		\label{eq:condExp1}
                \E^{\P}\left[\frac{d\Q}{d\P}\middle | \F_t\right]
                	 = {\rm d} \exp\left(-\int_0^t f(r, X_r)dr\right)\E^\P\left[\exp\left(- \int_t^T f(r, X_r)dr - g(X_T)\right)\middle | X_t\right],
	      \end{equation}
              	Combining \eqref{eq:condExp2} and \eqref{eq:condExp1} with \eqref{eq:formulaCondExp}, we get
	\begin{equation}
		\label{eq:firstEqQ}
		\E^{\Q}\left[F\left(\left(X_r\right)_{r \in [t, T]}\right)\middle | \F_t\right] = \frac{\E^{\P}\left[F\left(\left(X_r\right)_{r \in [t, T]}\right)\exp\left(- \int_t^T f(r, X_r)dr - g(X_T)\right)\middle | X_t\right]}{\E^\P\left[\exp\left(- \int_t^T f(r, X_r)dr - g(X_T)\right)\middle | X_t\right]}.
	\end{equation}
	The application of \eqref{eq:formulaCondExp} replacing $\shf_t$ with $X_t$ allows to conclude
        the proof.

\end{proof}

\subsection{The dynamics of the canonical process under the exponential twist measure}

We investigate now the dynamics of the canonical process $X$ under $\Q$, e.g. which martingale problem is fulfilled by $\Q$.

Since $\P$ verifies Hypothesis \ref{hyp:markovProp} (Markov property), for all $t \in [0, T]$, we have
\begin{equation}
	\label{eq:densityMarkovProp}
	\E^\P[D_T \vert \shf_t] = \frac{\exp\left(-\int_0^t f(r, X_r)dr\right) }{\E^\P\left[\exp\left(- \int_0^T f(r, X_r)dr - g(X_T)\right)\right]} \E^\P\left[\exp\left(-\int_t^Tf(r, X_r)dr - g(X_T)\right)\middle | X_t\right].
      \end{equation}
      We introduce now the useful notation
\begin{equation}
	\label{eq:definitionU}
	U^0_\cdot := \int_0^\cdot f(r, X_r)dr.
      \end{equation}
Below we define three significant processes playing
       a fundamental role in the sequel.
      
      \begin{notation} \label{not:MartD}
        \begin{enumerate}
\item We will denote by $D:= (D_t)_{t \in [0,T]}$ the càdlàg version of the martingale
$(\E^\P[D_T \vert \shf_t])$, see \eqref{eq:densityMarkovProp}.
\item 
      Below we define
\begin{equation} \label{eq:VD}
  V_t := \E^\P\left (\exp(- U^0_T - g(X_T))\right)
  \exp\left(U^0_t\right)D_t, \quad   \ t \in [0,T].
\end{equation}
\item We set
      \begin{equation}
		\label{eq:MVBis}
		M_t := V_t - V_0 - \int_0^t f(r, X_r) V_r dr, \ t \in[0,T].
              \end{equation}
\end{enumerate}
\end{notation}
  
\begin{remark} \label{rmk:MVD}
By \eqref{eq:densityMarkovProp} and \eqref{eq:VD} 
  we observe that
              $$ \E^\P(D_0) = 1, \E^\P(V_0) =  \E^\P(\exp(-U^0_T - g(X_T))), $$
              so that \eqref{eq:VD} becomes
            \begin{equation} \label{eq:VDBis}
	V_t = \E^\P(V_0)   \exp\left(U^0_t \right)D_t, \quad   \ t \in [0,T].
      \end{equation}
  By integration by parts, using \eqref{eq:VDBis} and \eqref{eq:MVBis}, we easily obtain that
	\begin{equation} \label{eq:MVTer}
		M_t =   \int_0^t \E[V_0]  \exp(U_r^0)dD_r, \ t \in [0,T].
	\end{equation}
        and
        \begin{equation} \label{eq:MVTerBis}
		D_t = D_0 +   \int_0^t \frac{1}{\E[V_0]}  \exp(-U_r^0)dM_r, \ t \in [0,T].
	\end{equation}
      \end{remark}

      \begin{remark} \label{rmk:Function_v}
 By Proposition 5.1 in \cite{MimickingItoGeneral}
there exists a Borel function $v : [0, T] \times \R^d \rightarrow \R$ such that
\begin{equation}
	\label{eq:representationV}
	v(t, X_t)
        = \E^\P\left[\exp\left(-\int_t^Tf(r, X_r)dr - g(X_T)\right)\middle | X_t\right]\quad dt \otimes d\P \text{-a.e.}
      \end{equation}

\end{remark}
\begin{definition} \label{definitionV} 
  A Borel function $v$ fulfilling \eqref{eq:representationV} 
  will be called {\bf intrinsic value function}.
  It is  uniquely defined as  element of $L^0(\P)$.
  Taking into account Remark \ref{rmk:Function_v}, $v$
  can be obviously chosen to be bounded and strictly positive. 
  \end{definition}

\begin{prop} \label{prop:MarkovMart}
  Let $\P$ be a probability measure fulfilling the Markov property, i.e. Hypothesis \ref{hyp:markovProp}. Let $v$
  be an intrinsic value function defined in Definition \ref{definitionV}.
       We have the following.      
          \begin{enumerate}
 \item  $V$ is a càdlàg version of $(v(t,X_t))$, i.e.
 \begin{equation} \label{eq:cadlagV}
   V_t = v(t,X_t), \quad dt  d\P\text{-a.e.} \quad
          \end{equation}
        \item $v \in \shd(\P)$ and $a^\P(v) = fv$.
          In particular $\shd(\P)$ is non trivial.
        \item Let $w \in \shd(\P)$ be another 
       solution of the deterministic problem
          $a^\P(w) = fw, w(T, \cdot) = e^{-g}$. Then $w  = v$ as an element of $L^0$.
\end{enumerate}
\end{prop}

\begin{proof}
  \begin{enumerate}
  \item	Combining \eqref{eq:densityMarkovProp} and
\eqref{eq:VD}, taking into account Remark \ref{rmk:Function_v},
    for almost all
    $t \in [0,T[,$ $\P$-a.s. we have
        \begin{equation} \label{eq:CondExp}
          v(t,X_t) = D_t \E^\P\left[\exp\left(- U^0_T - g(X_T)\right)\right] \exp\left(U^0_t \right)
          = V_t,
\end{equation}
which shows \eqref{eq:cadlagV}.
\item
 The process $M$ defined in 
		\eqref{eq:MVBis},
                by \eqref{eq:MVTer}  is a stochastic integral w.r.t. the martingale $D$, hence $M$ is a local martingale. Moreover, since it is the sum of a bounded process and a continuous adapted process (hence locally bounded), $M$ actually belongs to $\shh_{loc}^2(\P)$.

                By  Definition \ref{def:domainP} we get that $v \in \shd(\P)$.
            \item 	By assumption, for all $t \in [0, T]$ we have
	$$
	w(t, X_t) = w(0, X_0) + \int_0^t (fw)(r, X_r)dr + M[w]_t,
	$$
	where $M[w]$ is a local martingale. Then by integration by parts, the process
	$$
	\exp\left(-\int_0^t f(r, X_r)dr\right)w(\cdot, X_\cdot)
	$$
	is a local martingale, which is a genuine martingale since it is bounded. Consequently, by taking the conditional expectation with respect to $\shf_t$ and making use of the Markov property \eqref{eq:markovPropDef} we get
	$$
	w(t, X_t) = \E^\P\left[\exp\left(-\int_t^T f(r, X_r)dr - g(X_T)\right) \middle | X_t\right] \quad \P\text{-a.s.},
	$$
	for all $t \in [0, T]$. So $w$ verifies \eqref{eq:representationV},
        and by Remark \ref{rmk:v}, $w = v$ in $L^0$.
              \end{enumerate}

      \end{proof}
      \begin{remark} \label{rmk:MVBis}
        Obviously \eqref{eq:MVBis} implies that
        $M$ defined in \eqref{eq:MVTer} equals
        $M[v]$.
        
      \end{remark}
      We introduce now a clarifying property and
an equivalent useful characterization.
     The property below concerns $\P$ and a linear subspace
     $\shd \subset \shd(\P)$.

\begin{property}
	(Ideal Property).
	\label{cond:idealCondition}
	The element $v$ of $L^0$ introduced in Definition \ref{definitionV}
        is an element of $\shd(\P)$, such that $  \phi v \in \shd(\P)$ for all $\phi \in \shd$.
      \end{property}

      \begin{remark} \label{rkk:Ideal}
        We will see in Section \ref{sec:checkingHypV} that
   when $\P$ is Markovian, the Ideal Property is verified
   for any subalgebra
  $\shd \subset \shd(\P).$ 
  This will be the object of Theorem \ref{thm:IdealCond}.
\end{remark}

Next lemma provides some basic multiplication rules 
used in the sequel. It will help to find a useful
equivalent formulation for the Ideal Property.
Its proof is postponed to Section \ref{sec:proofsEquivalence}.

\begin{lemma}
	\label{lemma:equivalenceProduct}
	Let $\psi$ and $\phi$ be two elements of $\shd(\P)$. The following statements are equivalent.
	\begin{enumerate}
		\item $\phi \psi \in \shd(\P)$.
		\item 
                   The process (càdlàg modification of)
                  $(\phi \psi)(t,X_t)$
                  is locally square integrable under $\P$ and  there exists $\Gamma(\phi, \psi) \in \shb([0, T] \times \R^d, \R)$ (unique in $L^0$)
                  such that
		\begin{equation}\label{eq:bracket}
		\langle M[\phi], M[\psi] \rangle = \int_0^\cdot \Gamma(\phi, \psi)(r, X_r)dr \quad \P\text{-a.s.}
		\end{equation}
		Moreover we have
		\begin{equation}
			\label{eq:gammaPsiPhiP}
			\Gamma(\phi, \psi) = a^\P(\phi\psi) - \psi a^\P(\phi) - \phi a^\P(\psi).
		\end{equation}
	\end{enumerate}
\end{lemma}

\begin{remark}
  \label{rmk:carreDuChamp}
  \begin{enumerate}
    \item 
      The bilinear map $\Gamma$ in \eqref{eq:gammaPsiPhiP} is called the {\it carré du champ operator}.
    \item An immediate consequence of Lemma \ref{lemma:equivalenceProduct}
      is that, if $\phi \in \shd(\P)$ as well as $\phi^2$, then
      $ \langle M \rangle$ is absolutely continuous with respect to Lebesgue.
      \end{enumerate}
    \end{remark}

The corollary below is a consequence of Lemma \ref{lemma:equivalenceProduct}.
   \begin{coro}
	\label{coro:equivalenceProduct}
	Let $\shd \subset \shd(\P)$ be a linear subalgebra.
        The following statements are equivalent.
	\begin{enumerate}
        \item  $(\P, \shd)$ fulfills the Ideal Property.
	\item There exists a linear operator $\Gamma^v : \shd \rightarrow L^0$, such that for all $\phi \in \shd$,
		\begin{equation}
			\label{eq:bracketGamma}
			\langle M[\phi], M[v] \rangle = \int_0^\cdot
                        \Gamma^v(\phi)(r, X_r)dr.
		\end{equation}
		Moreover, 
		\begin{equation}
			\label{eq:operatorGammaPsi}
			\Gamma^v(\phi) = a^\P(v \phi) - v a^\P(\phi) -
                        \phi a^\P(v).
		\end{equation}
	\end{enumerate}
\end{coro}
\begin{proof}
  We apply Lemma \ref{lemma:equivalenceProduct} with
  $\psi = v$.
  We remark that, if $\phi \in \shd(\P)$ then, under $\P$,
the càdlàg modification $\Phi$ of  $(\phi(t,X_t))$ is locally square integrable
and therefore the càdlàg modification of $\phi(t,X_t) v(t,X_t)$, which is
 indistinguishable
  from the process $\Phi_t V_t$, is locally square integrable
  because the process $V$ defined in \eqref{eq:VD} is bounded.

  The result is then a direct consequence of Lemma \ref{lemma:equivalenceProduct}, the only thing to check
  being the linearity of $\Gamma^v$ which immediately follows from \eqref{eq:operatorGammaPsi} as $a^\P$ is linear. 
      \end{proof}

    Let us come back to the exponential twist measure $\Q$
defined in (\ref{eq:Qstar}), 
which
is equivalent to the reference probability measure $\P$,
whose characterization as solution of a martingale problem naturally requires the use of Girsanov's theorem. Notice first that $D \in \shh_{loc}^2(\P)$
since it is a bounded martingale taking into account
\eqref{eq:optimalDensity} and Hypothesis \ref{hyp:costFunctions}. In particular, for any $\shm \in \shh_{loc}^2(\P)$, $\langle \shm, D\rangle$ is well-defined under $\P$. 
Let us then recall the Girsanov's theorem in our context, see for example Theorem 3.11, Chapter III in \cite{JacodShiryaev} along with Proposition 3.5 item $(i)$, Chapter III in \cite{JacodShiryaev} for the positivity of $D$.
\begin{theorem}
	\label{th:girsa}(Girsanov).
	Let $\shm \in \shh_{loc}^2(\P)$.
        Let $\Q$ defined as in (\ref{eq:Qstar}) and $D$ be the strictly positive
        càdlàg martingale introduced as in Notation \ref{not:MartD}.

        Under $\Q$ the process
        $\shm - \int_0^\cdot \frac{1}{D_{r-}}d\langle \shm, D\rangle_r$ is a $\Q$-local martingale.
      \end{theorem}
    
   We state another preparatory lemma.

      \begin{lemma} \label{lm:girsa}
        Let $D$ be the càdlàg martingale defined in Notation \ref{not:MartD}. 
Let $v$ be a Borel function
introduced in Definition \ref{definitionV}.
Suppose the validity of the Ideal Property with respect to a linear subspace $\shd$
of $\shd(\P)$.
Let $\Gamma^v$ be the map defined in Corollary \ref{coro:equivalenceProduct}.
For every $\phi \in \shd$, we have  
\begin{equation}
	\label{eq:secondComputationsGirsaBis}
	\int_0^\cdot \frac{1}{D_{r-}} d\langle M[\phi], D \rangle _r=
	\int_0^\cdot \frac{\Gamma^v(\phi)(r, X_r)}{v(r, X_r)}dr,
      \end{equation}
where $M(\phi)$ for $\phi \in \shd(\P)$ was defined in
\eqref{eq:defMPhiP}
\end{lemma}

 \begin{proof}
  
  We need to evaluate $\langle M[\phi],D \rangle$;
   this, taking into account Remark \ref{rmk:MVBis},
and \eqref{eq:MVTer}, is equivalent to computing the bracket $\langle M[\phi], M[v] \rangle$.
  Indeed, since $D$ is strictly positive $\P$-a.s., the same holds for $V$ by \eqref{eq:VD}.
By \eqref{eq:MVTer} we have
$$
\langle M[\phi], M[v]\rangle = \int_0^\cdot \E^\P[V_0]\exp(U_r^0)d\langle M[\phi], D\rangle_r,
$$
so that
\begin{equation} \label{eq:BracketM}
\langle M[\phi], D\rangle = \int_0^\cdot \frac{\exp(-U_r^0)}{\E^\P[V_0]}d\langle M[\phi], M[v] \rangle.
\end{equation}
Consequently, by \eqref{eq:VDBis} and \eqref{eq:BracketM}
\begin{equation}
	\label{eq:firstComputationsGirsa}
	\int_0^\cdot \frac{1}{D_{r-}} d \langle M[\phi], D \rangle_r = \int_0^\cdot  \frac{\E^\P[V_0] \exp( U^0_r)}{V_{r-}}
        d\langle M[\phi], D \rangle_r 
      =  \int_0^\cdot \frac{d\langle M[\phi],
          M[v]\rangle_r}{V_{r-}}.
\end{equation}
      Since $v \in \shd(\P)$
and
$V_r = v(r,X_r), \ dr \otimes d\P$-a.e.,
by Corollary \ref{coro:equivalenceProduct},
 \eqref{eq:firstComputationsGirsa} becomes
\begin{equation}
	\label{eq:secondComputationsGirsa}
	\int_0^\cdot \frac{1}{D_{r-}}d\langle M[\phi], D \rangle _r=
        \int_0^\cdot \frac{\Gamma^v(\phi)(r, X_r)}{V_{r-}}dr =
	\int_0^\cdot  \frac{\Gamma^v(\phi)(r, X_r)}{V_{r}}dr =
	\int_0^\cdot \frac{\Gamma^v(\phi)(r, X_r)}{v(r, X_r)}dr.
\end{equation}
This concludes the proof of the lemma.
\end{proof}

About the martingale problem verified by $\Q$, we need to specify the linear operator $a^{\Q}$ and
     its domain, as we will do below in Theorem \ref{th:markovDrift}.
     The idea is to apply Theorem \ref{th:girsa} with $\shm = M[\phi]$
    and Lemma  \ref{lm:girsa}.

\begin{theorem}
	\label{th:markovDrift}
        Let our reference measure $\P$ verify Hypothesis \ref{hyp:markovProp}
         (Markov property).
Let $v: [0,T] \times \R^d \rightarrow \R$ be a Borel function
defined in Definition \ref{definitionV}.
 Let $v_0: \R \rightarrow \R$ be a Borel function such that
 $v_0(X_0) = \E[V_0 \vert X_0], \P$ a.s. 

Suppose $\P$ fulfills the Ideal Property with respect to a linear subspace
$\shd$ of  $\shd(\P)$.
Let $\mu$ be the law of $X_0$ under $\P$.
 Let  $\Q$ be the probability measure defined in (\ref{eq:Qstar}).

 Then $\Q$ is solution to the martingale problem associated to $(\shd, a^{\Q}, \nu)$ in the sense of Definition \ref{def:martingaleProblem}, where for all $\phi \in \shd$, we have
	\begin{eqnarray} 
          a^{\Q}(\phi)(t, x) &:=& a^\P(\phi)(t, x) +
      \frac{\Gamma^v(\phi)(t, x)}{v(t, x)} \label{eq:newGenerator}
       \\ 
          \nu(dx) &:=& v_0(x)\mu(dx)/\int_{\R^d}v_0(y)\mu(dy).
                       \label{eq:nudx}             
	\end{eqnarray}
\end{theorem}
\begin{remark} \label{rmk:DQstar}
  \begin{enumerate}
\item	In particular $\shd \subset \shd(\Q)$.
\item If $\shd = \shd(\P)$, then $\Q$ verifies the Markov
  property and $\shd(\P) \subset \shd(\Q)$.
  \end{enumerate}
\end{remark}
\begin{remark} \label{rmk:Good}
	As we mentioned in the Introduction, our Theorem \ref{th:markovDrift} has some similarities with 
	Theorem 4.2 from \cite{ExpoTwistMarkov} which supposes the existence of a 
        (so-called) \textit{''good function''} (according to Section 1 in \cite{ExpoTwistMarkov}) 
	$v : \R^d \mapsto \R^*_+$, which in particular belongs to $\shd(\P)$.
	\begin{enumerate}
        \item Our Theorem \ref{th:markovDrift} implies the result of
          Theorem 4.2 in \cite{ExpoTwistMarkov} under their assumption,
      at least under, the technical  hypotheses on $f := a^\P(v)/v $
                  and $g := -\log(v)$ to be lower bounded.
          In this case we are in position to apply our
          Theorem \ref{th:markovDrift},
          which entails the statement of Theorem 4.2 in \cite{ExpoTwistMarkov}. Indeed, the particular assumption ``$\shd^h_A = \shd(A)$`` of Theorem 4.2 in \cite{ExpoTwistMarkov} implies the validity of our Ideal Property for $(\P,\shd)$ when $\shd$ is the whole extended domain $\shd(\P)$.
		
        \item Theorem 4.2 in \cite{ExpoTwistMarkov} (stated in the time inhomogeneous setting) can be used to prove our Theorem \ref{th:markovDrift}.
          If we assume that the process in \cite{ExpoTwistMarkov} is of the form $(t,X_t)$ with time horizon $T$ and $X$ being an inhomogeneous Markov process,
          considerations just above \eqref{eq:representationV}  provide the existence of a good function $v$ on the basis of a running cost $f$ and a terminal cost $g$.
	\end{enumerate}

\end{remark}

\begin{proof}[Proof of Theorem \ref{th:markovDrift}.]
  We first check item 1. of Definition \ref{def:martingaleProblem}. Let $\psi \in \B_b(\R^d, \R)$. Then, taking into account
  \eqref{eq:densityMarkovProp}, we get
	\begin{equation*}
          \begin{aligned}
            \E^{\Q}[\psi(X_0)] = \E^\P[D_0\psi(X_0)] =
 \E^\P[\frac{V_0}{\E[V_0]} \psi(X_0)] =
            \frac{1}{\E^{\P}[v_0(X_0)]}\E^\P[v_0(X_0)\psi(X_0)].
		\end{aligned}
	\end{equation*}
	Hence $\mathcal{L}^{\Q}(X_0) = \nu,$ where
$\nu$ is defined in \eqref{eq:nudx}.
It remains to check item 2. of Definition \ref{def:martingaleProblem}. Let $\phi \in \shd$. Theorem \ref{th:girsa} states that under $\Q$ the process $M[\phi] - \int_0^\cdot \frac{1}{D_{r-}}d\langle M[\phi], D\rangle_r$ is a local martingale. Now by Lemma \ref{lm:girsa},
we have
$$
	\int_0^\cdot \frac{1}{D_{r-}}d\langle M[\phi], D \rangle_r = \int_0^\cdot \frac{\Gamma^v(\phi)(r, X_r)}{v(r, X_r)}dr,
	$$
	where $\Gamma^v$ is given by Corollary \ref{coro:equivalenceProduct}.
        Consequently, under $\Q$, the process
	\begin{equation*}
          \begin{aligned}
            M(\phi) &-\int_0^\cdot \frac{\Gamma^v(\phi)(r, X_r)}{v(r, X_r)}dr \\
		&=	\phi(\cdot, X_\cdot)  - \phi(0, X_0) - \int_0^\cdot a^\P(\phi)(r, X_r)dr - \int_0^\cdot \frac{\Gamma^v(\phi)(r, X_r)}{v(r, X_r)}dr\\
			& = \phi(\cdot, X_\cdot) - \phi(0, X_0) - \int_0^\cdot a^{\Q}(\phi)(r, X_r)dr
		\end{aligned}
	\end{equation*}
	is a local martingale. This concludes the proof.
\end{proof}

\subsection{Two significant particular cases}

Below we mention two special cases,
which will be explored in the sequel.
They are formulated in the hypothesis below which includes
two alternative items.

\begin{hyp}
  \label{hyp:existencePsxJumpDiffusion}
 Let $b \in \B([0, T] \times \R^d, \R^d)$, $\sigma \in \B([0, T] \times \R^d,
\R^{d \times d}).$  Let $L : [0, T] \times \R^d \times {\rm Bor}(\R^d)$
be a deterministic Lévy kernel in the sense of Definition \ref{def:levyKernel}. 

  \begin{enumerate}
  \item 
    We consider  a truncation function $k : \R^d \rightarrow \R^d$, i.e. a bounded real function defined on $\R^d$ equal to the identity in a neighborhood of zero.
  We suppose that $\P$ is solution to the martingale problem with respect to $(\shd, a, \mu),$ where $\mu \in \shp(\R^d)$, $\shd := \shd(a) := C_b^{1, 2}([0, T] \times \R^d, \R)$ and $a(\phi)$ is given as 
		\begin{equation}
			\label{eq:generatorJumps}
			\begin{aligned}
				a(\phi)(t, x) & = \partial_t \phi(t, x) + \langle \nabla_x\phi(t, x), b(t, x)\rangle + \frac{1}{2}Tr[\sigma\sigma^\top(t, x)\nabla_x^2 \phi(t, x)]\\
				& + \int_{\R^d}\left(\phi(t, x + q) - \phi(t, x) - \langle \nabla_x \phi(t, x), k(q)\rangle\right)L(t, x, dq),
			\end{aligned}
		\end{equation}
		for all $\phi \in \shd$.
\item We consider the same framework as in previous item 1.,
  replacing the non-local operator \eqref{eq:generatorJumps}
   with
	\begin{equation}
		\label{eq:generatorDiffBrownian}
		a(\phi)(t, x) = \partial_t \phi(t, x) + \langle \nabla_x\phi(t, x), b(t, x)\rangle + \frac{1}{2}Tr[\sigma\sigma^\top(t, x)\nabla_x^2 \phi(t, x)],
	\end{equation}
	for all $\phi \in \shd$.
\end{enumerate}
      \end{hyp}

      \begin{remark}\label{rmk:Jacod}
        \begin{enumerate} 
        \item        Theorem 2.42, Chapter II in \cite{JacodShiryaev} implies that a random measure $\nu$ is the compensator of the jump measure $\mu^X$
          if and only if,
for all $\phi \in C_b^2(\R^d)$ the process
\begin{equation*}
	\begin{aligned}
		& \phi(X_\cdot) - \phi(X_0) - \int_0^\cdot  (\nabla_x\phi)^\top (X_r) b(r, X_r) dr - \frac{1}{2}\int_0^\cdot Tr[\sigma\sigma^\top(r, X_r)\nabla_x^2 \phi(X_r)]dr \\
		& - \int_0^\cdot  \int_{\R^d}(\phi(X_{r-} + y) - \phi(X_r) - \langle \nabla_x\phi(X_r), y\rangle k(y))\nu (dr, dy),
	\end{aligned}
      \end{equation*}
      is a (locally square integrable) local martingale under $\P$. Indeed the characteristic triple is uniquely determined by previous property.
So, under the validity of  Hypothesis \ref{hyp:existencePsxJumpDiffusion} item 1., the Hypothesis \ref{hyp:compensator}
is fulfilled.
 \item
    Let us assume Hypothesis \ref{hyp:existencePsxJumpDiffusion} 2.
By   \eqref{eq:generatorDiffBrownian}, taking into account item 1. of the present Remark,  since the characteristics are uniquely determined,
we have $\nu^{X, \P} = 0$. This implies that the jump measure $\mu^X$ vanishes and the process $X$
is $\P$-a.s. continuous.
\end{enumerate}
\end{remark}

\section{Extension of the {\it carré du champ} under $\P$}
\label{sec:extension}
\setcounter{equation}{0}


In this Section \ref{sec:extension} we further characterize the operator $\Gamma^v$ appearing in Theorem \ref{th:markovDrift} and introduced in
Corollary \ref{coro:equivalenceProduct}.
Let now $\P$ be our reference  probability measure, fulfilling
the Markov Property Hypothesis \ref{hyp:markovProp}.
We consider a subdomain $\shd$ such that
$(\P,\shd)$ verifies the Ideal Property \ref{cond:idealCondition}.
Let $v:[0,T] \times \R^d \rightarrow \R$ be 
an intrinsic value function.

In this section, we will always also suppose the validity of Hypothesis \ref{hyp:compensator}.

  The proof of proposition below can be found in the Appendix,
see Section \ref{prop:GammaContinuousApp}.
\begin{prop}
  \label{prop:GammaContinuous}
  Let $\P \in \shp(\Omega)$ and let $\shd \subset \shd(\P)$
such that
$(\P,\shd)$ verifies the Ideal Property \ref{cond:idealCondition}.
We also suppose that $\P$ satisfies Hypothesis \ref{hyp:compensator}.
Given $ \phi \in  \shd$,
we denote
\begin{equation} \label{eq:martcomTer}
W(t, x, q) := (v(t, x + q) - v(t, x))(\phi(t, x + q) - \phi(t, x)).
\end{equation}  
Then we have the following.
        \begin{itemize}
        \item The process
          $ (\int_{\R^d} W(t, x, q) L(t, x, dq), t \in[0,T]),$
        is well-defined
        and it belongs to $\sha_{loc}(\P)$.
       \item
        Given $\Gamma^v$ introduced in Corollary \ref{coro:equivalenceProduct},
we define the linear operator
         $\Gamma^{v,c} : \shd \rightarrow \shb([0, T] \times \R^d, \R)$ by
	\begin{equation}
		\label{eq:decompGamma}
		\Gamma^{v,c} (\phi)(t, x) := \Gamma^{v}(\phi)(t, x) -
         \int_{\R^d} W(t, x, q) L(t, x, dq), \ (t,x) \in [0,T] \in \R^d,              \end{equation}
              for  $\phi \in \shd$.
        Then
	\begin{equation}
		\label{eq:defGammaContinuous}
		[ M[v]^c, M[\phi]^c] = \int_0^\cdot\Gamma^{v, c}(\phi)(r, X_r)dr
                 \quad \P \text{-a.s.},
               \end{equation}
             \end{itemize}
             where, given a local martingale $M$, $M^c$
             indicates the continuous local martingale component of $M$.
      \end{prop}

      \begin{remark} \label{rmk:GammaContinuous}
        The identity  \eqref{eq:defGammaContinuous}
        also shows that       $ \Gamma^{v,c}$
        can be considered as a map $ \shd \rightarrow L^0$.
      \end{remark}

Proposition \ref{prop:GammaContinuous} states that, under Hypothesis \ref{hyp:compensator},
the operator $\Gamma^v$ can be decomposed into two components: the first one is a component $\Gamma^{v, c}$,
given by \eqref{eq:decompGamma}, related to covariations of continuous local martingales and
the second one  is related to the  jumps compensation. 

We extend below the covariations of semimartingales and we introduce the notion
of weak Dirichlet process.

\begin{definition}
  \label{def:weakDirichlet}
   Let $\P \in \shp(\Omega)$.
        Let $Y, Z$ be  a c\`adl\`ag process.
        \begin{enumerate}
          \item (Covariation).
        We define
\begin{equation} 
\label{Appr_cov_ucpI}	 
[Z,Y]^{\varepsilon}(t):= \,\int_{]0,\,t]}\,
\frac{(Z((r+\varepsilon)\wedge t)-Z(r))(Y((r+\varepsilon)\wedge t)-Y(r))}{\varepsilon}\,dr.
\end{equation}
$[Z,Y]$ is  by definition the u.c.p. limit,
whenever it exists, of $[Z,Y]^{\varepsilon}$ when $\epsilon \rightarrow 0$.
If $Y, Z$ are c\`adl\`ag semimartingales
then $[Y,Z]$ is the usual (quadratic) covariation, see Proposition 1.1 of
\cite{rv95}.
\item (Weak Dirichlet process).
  $Z$ is called a {\bf weak Dirichlet process} if it is $(\F_t)$-adapted and if under $\P$ it admits a decomposition $Z = M + A$, where $M$ is a $(\P, \shf_t)$-local martingale and the process $A$ satisfies $[A, N] = 0$ for all $(\P, \F_t)$-continuous local martingales. $A$ will be called a \textbf{martingale orthogonal process}. For more properties on those processes, see
  \cite{Russo_Vallois_Book}, Chapter 15 and \cite{WeakDirichletJumps, GeneralizedMP}. In particular an $(\shf_t)$-semimartingale is a weak Dirichlet process.
\item A multidimensional weak Dirichlet process is a multidimensional process
  such that every component is a weak Dirichlet process.
\item If $Y, Z$ are vector-valued process (considered as column vectors)
  then the matrix $[Y,Z]:=([Y,Z]_{ij})$ denotes the matrix-valued process
  $([Y^i,Z^j])$.
\end{enumerate}
\end{definition}
The following statement is   Proposition 3.2 in \cite{GeneralizedMP}.
\begin{prop}
	\label{prop:decompWeakDirichlet}
	Let $Z$ be a c\`adl\`ag weak Dirichlet process. There exists a unique continuous local martingale $Z^c$ and a unique process $A,$ vanishing at zero, verifying $[A, N] = 0$ for all $(\P, \F_t)$-continuous local martingale such that $Z = Z^c + A$.
      \end{prop}
    We go on now focusing on a better characterization of the
      map $\Gamma^{v,c}$ when $X$ is weak Dirichlet process
      with continuous martingale component of diffusive type.
      This will include a large class of Markov processes
      even with very irregular drift so that they are not even semimartingales.

      \begin{coro} \label{cor:P1PDE} 
        We suppose the following for the basic Markovian reference probability $\P$.
        \begin{itemize}
        \item $\P$ fulfills a martingale problem with respect to $(\shd, a, \mu)$ with some operator
        $a$, an initial condition $\mu$ and $\shd \subset C^{0,1}$, in the sense of Definition \ref{def:martingaleProblem},
        \item Under $\P$, the canonical process $X$ is a weak Dirichlet process with unique decomposition $X = X^c + A$
         and there is a locally bounded function
         $\sigma: [0,T] \times \R^d\rightarrow \R^d$,
         with $[X^c, X^c]_{\cdot} = \int_0^\cdot \sigma \sigma^\top(s,X_s)ds.$
       \item Under $\P$, $X$ is a weakly finite quadratic variation process,
         see Definition 3.30 of \cite{GeneralizedMP}.
       \item $\P$ verifies Hypothesis \ref{hyp:compensator}.
\item There is a
 non-negative function $w \in C^{0,1}([0,T]\times \R^d)$,
     belonging to $\shd(\P)$ and verifying $a^\P(w) = fw.$
\end{itemize}
     Then $w$ is an $L^0$-version of the intrinsic value function $v$
	which is defined in Definition \ref{definitionV}. 
        Moreover the map $\Gamma^{v,c}$ defined in Proposition
        \ref{prop:GammaContinuous}
        can be characterized as 
        \begin{equation} \label{eq:gradv}
        \Gamma^{v,c}(\phi) := (\nabla_x \phi)^\top \sigma \sigma^\top\nabla_x w.
\end{equation}
\end{coro}
\begin{remark}
  \begin{enumerate}
  \item For $\varepsilon > 0$, We set
$$ Z(\varepsilon) = 
          \int_0^T \frac{\vert X_{(r+\varepsilon)\wedge T} -  X_r\vert^2}{\varepsilon}dr.$$ According to Proposition 3.32 of \cite{GeneralizedMP}
          $X$ (under $\P$), is a weakly finite quadratic variation process
          if one of the two following conditions are fulfilled.
          \begin{itemize}
            \item  $ \sup_{0 < \varepsilon \le 1} Z(\varepsilon)
            < +\infty, \P$-a.s. 
          \item $  \sup_{0 < \varepsilon \le 1} \E^{\P}(Z(\varepsilon)) < + \infty.$
            \end{itemize}
          \item Obviously, if $[X,X]$ exists, then
            $X$ is a weakly finite quadratic variation process. 
            \end{enumerate}
          \end{remark}

\begin{proof} [Proof of Corollary \ref{cor:P1PDE}]
By Proposition \ref{prop:MarkovMart}, $w = v$, as element of $L^0$.
By Theorem 3.37 of \cite{GeneralizedMP} we have
\begin{eqnarray*}
M[w]^c &=& w(0,X_0) + \int_0^\cdot (\nabla w)^\top(r,X_r) dX^c_r, \\
M[\phi]^c &=& \phi(0,X_0) +  \int_0^\cdot (\nabla \phi)^\top(r,X_r) dX^c_r.
  \end{eqnarray*}
  Consequently
$$   [ M[v]^c, M[\phi]^c] = [ M[w]^c, M[\phi]^c] = \int_0^\cdot (\nabla_x \phi)^\top \sigma \sigma^\top\nabla_x w(r,X_r) dr.$$
   The result follows by  Proposition \ref{prop:GammaContinuous}
which states that
$$		\int_0^\cdot\Gamma^{v, c}(\phi)(r, X_r)dr =
		[ M[v]^c, M[\phi]^c]. $$
              \end{proof}

We now extend this new operator $\Gamma^{v, c}$, defined in Proposition  \ref{prop:GammaContinuous}, 
from $\shd$ (being the domain of a martingale problem)
to a wider domain $\scrd$.
By assumption, for every $\phi \in \shd,$
$\phi(\cdot, X_\cdot)$ is a special semimartingale.
This will  allow to identify a unique decomposition for $\phi(\cdot, X_\cdot)$ 
even for an important class of non-special semimartingales.

The extension of $\Gamma^{v, c}$ will 
naturally intervene in the formulation of  the martingale problem verified by
$\Q$ in the examples in Section \ref{sec:applications} when the process $X$ is not a special semimartingale under $\P$. 
For $\phi, \psi \in \shd(\P)$, let $d_c$ be defined by
\begin{equation}
	\label{eq:metricD2}
	d_c(\phi, \psi) := \E^\P\left[\frac{[M[\phi]^c - M[\psi]^c ]_T}{1 + [ M[\phi]^c - M[\psi]^c ]_T}\right].
\end{equation}
\begin{remark} \label{rmk:pseudometric}
    \begin{enumerate}
  \item The application $d_c$ introduced by \eqref{eq:metricD2} is a {\it semidistance} in the sense that it is non-negative, symmetric, verifies the triangular inequality but $d_c(\phi, \psi)$ might be $0$ even if $\phi \neq \psi$.
  \item $d_c$ is homogeneous in the sense that
    $d_c(\phi, \psi) =  d_c(\phi-\psi, 0)$ for all $\phi, \psi \in  \shd(\P)$.
    \end{enumerate}
      \end{remark}
      We endow $L^0$ defined in Notation \ref{rmk:v},
      with the natural metric
\begin{equation}
	\label{eq:d3}
	d_{L^0}(\phi, \psi) := \E^\P\left[\frac{\int_0^T |\phi - \psi|(r, X_r)dr}{1 + \int_0^T |\phi - \psi|(r, X_r)dr}\right].
\end{equation}

At this point we can extend the operator $\Gamma^{v,c}$ from $\shd$  naturally to a larger space.

\begin{definition}
	\label{def:closureDomain}
	(Closure of $\shd$).
	A linear metric space $(\scrd, d_\scrd),$  where
$d_\scrd$ is a homogeneous distance,
        is said to be a closure of $\shd$ if the following holds.
	\begin{enumerate}
		\item $\shd$ is dense in $\scrd$ with respect to the metric $d_\scrd$.
		\item $d_c + d_{L^0} < d_\scrd$ on $\shd$, in the sense that convergence under $d_\scrd$ implies convergence under $d_c + d_{L^0}$.
		\item For every $\phi \in \scrd$ and for every continuous $\P$-local martingale $N$,
                  the covariation  $[\phi(\cdot, X_\cdot), N]$ exists.
	\end{enumerate}

      \end{definition}

\begin{remark}
  \label{rmk:equivalence}
    Let $(\scrd, d_{\scrd})$ be the closure of $(\shd, d_{\scrd})$
in the sense of Definition \ref{def:closureDomain}.
  \begin{enumerate}
  \item 
An immediate consequence of item 3. above, is that 
$\phi \mapsto [\phi(\cdot, X_\cdot), N]$ is continuous from $\scrd$ to $\D^{ucp}$ with respect to the metric $d_\scrd$ for every continuous $\P$-local martingale $N$.
This follows easily by Banach-Steinhaus theorem for $F$-spaces, see e.g. Chapter 2.1 in \cite{dunford}, taking into account Definition \ref{def:weakDirichlet} item 1. 
\item A sufficient condition for the validity of item 3. above, 
  still making use of
  the same Banach-Steinhaus theorem, is that
  for every  $\phi \in \scrd$, the process $\phi(\cdot,X)$ is a weakly finite quadratic variation process. Indeed,
this implies that, fixing a continuous local martingale $N$,
  the family $[\phi(\cdot, X_\cdot),N]^{\varepsilon}$ is bounded in
the $F$-space $\D^{ucp}$.
  Moreover
  $\shd$ is a dense subset of $\scrd$ and the covariation $[\phi(\cdot, X_\cdot), N]$ always exists when
  $\phi \in \shd$, since $\phi(\cdot, X_\cdot)$ is a semimartingale.
  
  \end{enumerate}
        \end{remark}

      The proof of the proposition below is in the
      Appendix, see Section \ref{prop:closureWeakDirichletApp}.
\begin{prop}
	\label{prop:closureWeakDirichlet}
	Let $(\scrd, d_\scrd)$ be a closure of $\shd$ in the sense of Definition \ref{def:closureDomain}. Let $\phi \in \scrd$. Then $\phi(\cdot, X_\cdot)$ is a weak Dirichlet process in the sense of Definition \ref{def:weakDirichlet}.
      \end{prop}

      \begin{remark} \label{rmk:closure}
  	Let $(\scrd, d_\scrd)$ be a closure of $\shd$ in the sense of Definition \ref{def:closureDomain}. 
        From now on, if $\phi \in \scrd$, 
        $M[\phi]^c$ will denote the unique continuous local martingale
of the weak Dirichlet decomposition of $\phi(\cdot, X_\cdot)$, see Proposition \ref{prop:decompWeakDirichlet}. 
\end{remark}

\begin{example}
  \label{ex:closure}
  Below we provide two examples, the first one (see Proposition
\ref{prop:closureCont})
  in a non-semimartingale framework,
  the second in a bounded variation (purely jump) context,
  see Proposition \ref{prop:JumpExamples}.

  \begin{prop} \label{prop:closureCont}
    Let $\shd = \shc_b^{1, 2}$.  Consider $\scrd = \shc^{0, 1}$, equipped with
    the metric $\scrd$ so that a sequence 
$(f_n)$ converges to $f$ in $\scrd$ if it converges uniformly on compacts sets
     to $f$, as well as
    the corresponding space derivatives.
          We suppose the following. 
          \begin{enumerate}
  \item   Hypothesis \ref{hyp:compensator}.
          \item 
         The canonical process $X$ (under $\P$) is a weakly finite quadratic variation process.
         \end{enumerate}
         Then $\scrd$ is a closure of $\shd$
         in the sense of Definition \ref{def:closureDomain}.
     \end{prop}
\begin{proof}[Proof of Proposition \ref{prop:closureCont}]
       
  Obviously $\shd$ is dense in $\shc^{0, 1}$ which establishes item 1. of Definition \ref{def:closureDomain}.
  Concerning item 2., we  first observe that $d_{L^0} < d_\scrd$. 
On  the other hand,
  for all $\phi \in \shd$,
  $\phi(\cdot, X_\cdot)$ is a special semimartingale, and in particular a weak Dirichlet process.
          By Theorem 4.3 in \cite{GeneralizedMP}, under $\P$ the canonical process $X$ is a weak Dirichlet process.
          Let  $X = M^c + A$ be the unique decomposition under $\P$ according to Proposition \ref{prop:decompWeakDirichlet}.
          By Theorem 3.37 in \cite{GeneralizedMP}, the unique continuous local martingale part $M[\phi]^c$ of $\phi(\cdot, X_\cdot)$ is 
\begin{equation} \label{eq:MPhic}
          M[\phi]^c = \int_0^\cdot (\nabla_x\phi)^\top(r, X_r)dX_r^c.
        \end{equation}
            Concerning the fact that  $d_c < d_\scrd$, for all $\phi \in \shd$, \eqref{eq:MPhic} implies
		\begin{equation*}
[ M[\phi]^c ]_T = \left[\int_0^\cdot (\nabla_x\phi)^\top(r, X_r) dX_r^c\right]_T = \sum_{i, j = 1}^d \int_0^T \partial_{x_i}\phi(r, X_r)\partial_{x_j}\phi(r, X_r)d[X^{c, i}, X^{c, j}]_r.
		\end{equation*}
		Now if $\phi_n \underset{n \rightarrow + \infty}{\longrightarrow} 0$ in $\shc^{0, 1}$, clearly $[ M[\phi]^c]_T \underset{n \rightarrow + \infty}{\longrightarrow} 0$ in probability under $\P$ and it follows that $d_c(\phi_n, 0) \underset{n \rightarrow + \infty}{\longrightarrow} 0$.
Consequently we also have $d_{c} < d_\scrd$, which implies item 2. 
                
                For any $\phi \in \shc^{0, 1}$, $\phi(\cdot, X_\cdot)$ being a weak Dirichlet process, taking into account \eqref{eq:MPhic} we have
		$$
		[\phi(\cdot, X_\cdot), N] = [M[\phi]^c,N] = \int_0^\cdot (\nabla_x \phi)^\top(r, X_r)d[X^c, N].
		$$
		So, previous equality implies that the map  $\phi \mapsto [\phi(\cdot, X_\cdot),N]$ is continuous,
which yields item 3. of Definition \ref{def:closureDomain}.
                 We conclude from the above that $\scrd$ is a closure of $\shd$.

                 We remark that Theorems 3.37 and 4.3 in \cite{GeneralizedMP} are stated in the one-dimensional
                 framework but they can be easily extended to the
                 multidimensional case.
\end{proof}

The processes described below are motivated by the framework of Piecewise Deterministic
  Markov Processes (PDMP), see e.g. Chapter 2. of
\cite{Da-bo}, sections 24 and 26 or
Section 4.3  in \cite{BandiniRusso2}.

\begin{prop}  \label{prop:JumpExamples}
  Let $(T_n)_{n \geq 0}$ be a strictly increasing sequence of
  stopping times, such that $T_n \uparrow \infty$ a.s. Let $X$ be a (c\`adl\`ag) process of the type
	$$
	X_t = \sum_{n = 1}^\infty \one_{[T_{n-1}, \,T_n[}(t)\, R_{n-1}(t), 
	$$
	where $R_n$ is an $(\mathcal F_{T_n}), n \in \N,$-measurable c\`adl\`ag process.

        We set here   $\shd = \shc_b^{1, 2}$
          and $\scrd = \shc^{0, 0}$, equipped with a metric   $d_\scrd$
                compatible with the uniform convergence on compact sets.
 
                Then, again, $\scrd$ is a closure of $\shd$
         in the sense of Definition \ref{def:closureDomain}.
       \end{prop}

       \begin{proof} \
         We need to check the items of  Definition \ref{def:closureDomain}.
         Again item 1.  is fulfilled.

         If $\phi$ is any continuous function,
        we know by Proposition 5.38 of \cite{WeakDirichletJumps} that
         $(\phi(t, X_t))$ is a martingale orthogonal process.
         By  Proposition \ref{prop:decompWeakDirichlet},
         $M[\phi]^c_t \equiv M[\phi]^c_0$, which implies
          that $d_c \equiv 0$.
         Since trivially  $d_{L^0} < d_\scrd$, item 2. is fulfilled.


         Item 3. is obviously verified,  since $\phi(\cdot, X_\cdot)$
         is a martingale orthogonal process. 
\end{proof}

\end{example}

Before proving the main result of this section, we extend the operator $\Gamma^{v, c}$ introduced in Proposition \ref{prop:GammaContinuous} from $\shd$ to $\scrd$. The proof of the result below will be formulated in the Appendix,
see Section \ref{prop:extensionGammaClosureApp}.
\begin{prop}
	\label{prop:extensionGammaClosure}
	Let $\shd \subset \shd(\P)$ a subalgebra and $(\scrd, d_{\scrd})$ be a closure of $\shd$.
        Assume moreover that $\P$ verifies Hypothesis \ref{hyp:compensator} and let $\Gamma^{v, c}$ be the operator given by Proposition \ref{prop:GammaContinuous}.
	The operator $\Gamma^{v, c} : \shd \rightarrow L^0$ extends continuously to $\scrd \rightarrow L^0$: we will keep the notation $\Gamma^{v, c}$ for the extension and we still have
        	\begin{equation}
		\label{eq:defGammaContinuousExt}
		\int_0^\cdot\Gamma^{v, c}(\phi)(r, X_r)dr =
		[ M[v]^c, M[\phi]^c]  \quad \P \text{-a.s.}, \ \forall \phi \in \scrd.
	\end{equation}
      \end{prop}
      \begin{remark} \label{rmk:ClosureIdeal}
        Since $\P$ verifies the Markov property and $\shd$ is a subalgebra
        of $\shd(\P)$ then the Ideal Property  \ref{cond:idealCondition}
        is fulfilled, by Theorem \ref{thm:IdealCond} in the next section.
       \end{remark}

      Below we have finally the most important result of the section.
\begin{prop}
	\label{prop:idGamma}
	Let us assume the following for our reference probability $\P$.
\begin{itemize}
\item Under $\P$, the canonical process $X$ is a weak Dirichlet process with unique decomposition $X = X^c + A$ given by
  Proposition \ref{prop:decompWeakDirichlet}.
  \item $X$ is a weakly finite quadratic variation process.
\item Hypothesis \ref{hyp:compensator}.
\item $\shd$ is dense in $\scrd = \shc^{0, 1}$.
\end{itemize}
Then $\scrd$ is a closure of $\shd$ in the sense of Definition \ref{def:closureDomain} and for all $\phi \in \scrd$,
        $(t, x) \in [0, T] \times \R^d$,
	\begin{equation}
		\label{eq:idGamma}
		\Gamma^v(\phi)(t, x) = \langle \Gamma^{v, c}(id)(t, x), \nabla_x \phi(t, x) \rangle + \int_{\R^d}(v(t, x + q) - v(t, x))(\phi(t, x + q) - \phi(t, x))L(t, x, dq), 
	\end{equation}
	with $\Gamma^{v, c}(id) := \left(\Gamma^{v, c}(id_i)\right)_{1 \le i \le d}$, $\Gamma^{v, c}$ being the linear operator given by Proposition \ref{prop:extensionGammaClosure}.
      \end{prop}

\begin{proof}

  The fact that $\scrd$ is a closure of $\shd$ was the object  of
Proposition \ref{prop:closureCont}.
Concerning the proof of \eqref{eq:idGamma}, we see from \eqref{eq:decompGamma} in Proposition \ref{prop:GammaContinuous} that it is enough to show that
	\begin{equation}
		\label{eq:idGammaC}
		\Gamma^{v, c}(\phi)(t, x) = \langle \Gamma^{v, c}(id)(t, x), \nabla_x \phi(t, x) \rangle \quad dt\otimes d\P_{X_t}\text{-a.e.}
	\end{equation}
         Recall that for all $\phi \in \scrd$,
        by \eqref{eq:defGammaContinuousExt} in Proposition \ref{prop:extensionGammaClosure},
        we have
\begin{equation} \label{eq:partialvId}
  \int_0^\cdot \Gamma^{v, c}(\phi)(r, X_r)dr = [ M[v]^c, M[\phi]^c].
 	\end{equation}
In particular, taking $\phi = id_i$, which is an element of
        $\scrd$,
\begin{equation} \label{eq:partialvIdPlus}
  \int_0^\cdot \Gamma^{v, c}(id)(r, X_r)dr = [ M[v]^c, M[id]^c] =
   [ M[v]^c, X^c],
 	\end{equation}
where $M[id]:= (M[id_i])_{1 \le i \le d}$.
        Let $\phi \in \shd$.
        Taking into account the definition of $L^0$,
        it is enough to prove that
	\begin{equation} \label{eq:enough}
			[ M[v]^c, M[\phi]^c]  =  \int_0^\cdot (\nabla_x\phi)^\top(r, X_r)\Gamma^{v, c}(id)(r, X_r)dr.
		\end{equation}
                Now, since $X$ is a weak Dirichlet process under $\P$ and it is a weakly finite quadratic variation process,
                Theorem 3.37 in \cite{GeneralizedMP} yields
	\begin{equation} 
          \label{eq:MPhicDirichlet}
		M[\phi]^c = \phi(0,X_0) + \int_0^\cdot (\nabla_x \phi)^\top(r, X_r) dX^c_r.
	\end{equation}
Finally 
	\begin{equation*}
			[ M[v]^c, M[\phi]^c]  = \int_0^\cdot (\nabla_x \phi)^\top(r, X_r)d[M[v]^c, X^c]_r
	\end{equation*}
gives \eqref{eq:enough} using \eqref{eq:partialvIdPlus}.

\end{proof}



      \section{The verification of the Ideal Property}
      
      \label{sec:checkingHypV}

      

\setcounter{equation}{0}

As announced earlier we show in this section that a
Markovian probability often fulfills the Ideal property.
The most important result of the section is a following one.

\begin{theorem} \label{thm:IdealCond}
  Let $\P$ be a probability fulfilling Hypothesis \ref{hyp:markovProp}.
  Then $\P$  verifies the Ideal Property
  with respect to every subalgebra $\shd$ of $\shd(\P)$.
\end{theorem}

  The proof of Theorem \ref{thm:IdealCond} will be based on the proposition
  below.
  \begin{prop}
		\label{prop:exitenceGamma}
		Let $M, N$ be two elements of $\shh_{loc}^2(\P)$ such that $M_u - M_t$ and $N_u - N_t$ are $\shf_{t, u}$-measurable for all $0 \le t \le u \le T$. Assume that $d\langle N \rangle \ll dt$. There exists a Borel function $\Gamma : [0, T] \times \R^d \rightarrow \R$ such that
		\begin{equation}
			\label{eq:angleBracketGamma2}
			\langle M, N\rangle = \int_0^\cdot \Gamma(r, X_r)dr \quad \P\text{-a.s.}
		\end{equation}
              \end{prop}

  \begin{proof} [Proof of Theorem \ref{thm:IdealCond}]
    Let $\phi \in \shd$. Since $\shd$ is an algebra
    also   $\phi^2 \in \shd \subset  \shd(\P)$. By
     Remark \ref{rmk:carreDuChamp},
     the angle bracket of the
    locally square integrable martingale $M(\phi)$  is absolutely continuous.
      Since $v \in \shd(\P)$, then $M[v]$ is a locally square integrable
    martingale and so, we can apply Proposition
    \ref{prop:exitenceGamma}
with $M = M[v]$ and $N = M[\phi]$, 
    to conclude the proof.

  \end{proof}

     Before formulating the proof of Proposition \ref{prop:exitenceGamma},
     we need to establish some intermediate results.

	\begin{lemma}
		\label{lemma:piSystem}
		For all $t \in [0, T]$, $\Pi := \{F_t \cap F_{t, T}~:~F_t \in \shf_t, F_{t, T} \in \shf_{t, T}\}$ is a non-empty $\pi$-system (i.e.
                a family  stable with respect to the intersection)
               such that $\sigma(\Pi) = \shf$.
	\end{lemma}
	\begin{proof}
          We only  prove that  $\sigma(\Pi) = \shf$, the rest being
          straightforward
          taking into account the Dynkin monotone class theorem,
          see e.g. Theorem 3.2, Chapter 1, in \cite{billingsley}.
Since $\shf_{t, T}$ is a $\sigma$-field, $\Omega \in \shf_{t, T}$ and for all $F_t \in \shf_t$, $F_t = F_t \cap \Omega \in \Pi$. Hence $\shf_t \subset \Pi$.
Similarly, $\shf_{t, T} \subset \Pi$ and we have $\shf = \shf_t \vee \shf_{t, T} \subset \sigma(\Pi)$. Since $\sigma(\Pi) \subset \shf$, we conclude that $\sigma(\Pi) = \shf$.
\end{proof}
\begin{remark}
	\label{rmk:piSystem}
	The result of Lemma \ref{lemma:piSystem} remains valid replacing $\shf_t$ by $\shf_{t, r}$, $\shf_{t, T}$ by $\shf_{r, T}$ and $\shf$ by $\shf_{t, T}$,
        for all $r \in [t,T]$.
      \end{remark}

	\begin{prop}
		\label{prop:FtuMeasurable}
		Let $M \in \shh_{loc}^2(\P)$ such that $M_u - M_t$ is $\shf_{t, u}$-measurable for all $0 \le t \le u \le T$.
 $\langle M\rangle_u - \langle M \rangle_t$ is $\shf_{t, u+}^\P$-measurable for all $0 \le t \le u \le T$.
\end{prop}
Before proving Proposition \ref{prop:FtuMeasurable}, we mention that,
	in what follows, the conditional expectation of non-negative random variables has to be understood in the generalized sense given in Proposition \ref{prop:generalizedCondExp}. It is necessary, in our setting, since for example the random variable $\langle M \rangle_u - \langle M \rangle_t$ might not be integrable. Yet this version of the conditional expectation has the same characterization as the usual conditional expectation and the Markov property still holds, see Proposition \ref{prop:generalizedCondExp} and Proposition \ref{prop:markovPropGen}. The proof of Proposition \ref{prop:FtuMeasurable} is inspired by the proof of Proposition 4.5 in \cite{MartingaleAF} and requires several intermediate steps. In the following results, $M \in \shh_{loc}^2(\P)$ is such that $M_u - M_t$ is $\shf_{t, u}$-measurable for all $0 \le t \le u \le T$.
	\begin{lemma}
		\label{lemma:quadraticVariation}
		For all $0 \le t \le u \le T$, $[M]_u - [M]_t$ is $\shf_{t, u}^\P$-measurable.
	\end{lemma}
	\begin{proof}
		Let $t = t_1^n < t_2^n < \ldots < t_n^n = u$ be a sequence of subdivisions of the interval $[t, u]$ such that $\underset{i < n}{\max}~(t_{i + 1}^n - t_i^n) \underset{n \rightarrow + \infty}{\longrightarrow} 0$. By definition of the quadratic variation, see e.g. Theorem 4.47, Chapter I in \cite{JacodShiryaev}, we have
		\begin{equation}
			\label{eq:quadraticVarLim}
			\sum_{i < n}\left(M_{t_{i + 1}^n} - M_{t_{i}^n}\right)^2 \overset{\P}{\underset{n \rightarrow + \infty}{\longrightarrow}} [M]_u - [M]_t,
		\end{equation}
		and since the random variable $\sum_{i < n}\left(M_{t_{i + 1}^n} - M_{t_{i}^n}\right)^2$ is $\shf_{t, u}$-measurable, $[M]_u - [M]_t$ is $\shf_{t, u}^\P$-measurable.
	\end{proof}
        
	\begin{lemma} \label{lm:condExpFtXt}
		Let $0 \le t \le u \le T$. For all $F \in \shf_{t, T}$,
		\begin{equation}
			\label{eq:condExpFtXt}
			\E^\P\left[\1_F\left(\langle M \rangle_{u} -\langle M \rangle_{t}\right) |  \shf_t\right] = \E^\P\left[\1_F\left(\langle M \rangle_{u} -\langle M \rangle_{t}\right) |  X_t\right] \quad \P\text{-a.s.}
		\end{equation}
	\end{lemma}
	\begin{proof}
		\begin{enumerate}
			\item Let for now $F \in \shf$ and $N^F$ be the c\`adl\`ag version of the martingale $r \mapsto \E^\P[\1_F|\F_r]$. Let us first prove that
			\begin{equation}
				\label{eq:condExpStep2}
				\E^\P\left[\1_F\left(\langle M \rangle_{u} -\langle M \rangle_{t}\right) | \shf_t\right] = \E^\P\left[\int_t^u N_{r-}^Fd[M]_r \middle | \shf_t\right] \quad \P\text{-a.s.}
			\end{equation}
			Recall that the process $\langle M \rangle$ is an increasing, locally integrable process. Then since $N^F$ is bounded and since $\langle M \rangle$ is predictable, by Proposition 3.14 item b), Chapter I in \cite{JacodShiryaev}, it holds that the process $\tilde M := N^F\langle M \rangle - \int_0^\cdot N_{r-}^Fd\langle M \rangle_r$ is a local martingale. Let then $(\tau_n)_{n \ge 0}$ be a localizing sequence for $\tilde M$ and $\langle M \rangle$. Then $\E^\P\left[\tilde M_{u \wedge \tau_n}\right] = 0$, which rewrites
			\begin{equation}
				\label{eq:firstEqDual}
				\E^\P\left[\1_F\langle M \rangle_{u\wedge \tau_n}\right] = \E^\P\left[\int_0^{u \wedge \tau_n}N_{r-}^Fd\langle M \rangle_r\right],
			\end{equation}
			where we have used the tower property of the conditional expectation for the left-hand side of equality \eqref{eq:firstEqDual}.
  Concerning  the right-hand side of \eqref{eq:firstEqDual},
 we remark that the r.v. inside the expectation is integrable
 since $N^F$ is bounded and $\langle M \rangle_{\cdot \wedge \tau_n}$ is
 integrable.
 We recall that, given the local martingale $M \in \shh_{loc}^2(\P)$,
 the oblique bracket $\langle M \rangle$ is
 the compensator of $[M]$,
 see for example 3.20, Chapter I in \cite{JacodShiryaev}.
       The non-negative process $r \mapsto N_{r-}^F$ being càglàd and
        adapted, is predictable. It follows then
       from Theorem 3.17 item (iii), Chapter I in \cite{JacodShiryaev} that
			$$
                        \E^\P\left[\int_0^{u \wedge \tau_n}N_{r-}^Fd\langle M \rangle_r\right] = \E^\P\left[\int_0^{u \wedge \tau_n}N_{r-}^Fd[M]_r\right],
			$$
			which, by \eqref{eq:firstEqDual}, yields
			$$
			\E^\P\left[\1_F\langle M \rangle_{u\wedge \tau_n}\right] = \E^\P\left[\int_0^{u \wedge \tau_n}N_{r-}^Fd[M]_r\right].
			$$
			Similarly,
			$$
			\E^\P\left[\1_F\langle M \rangle_{t\wedge \tau_n}\right] = \E^\P\left[\int_0^{t \wedge \tau_n}N_{r-}^Fd[M]_r\right],
			$$
			and it follows from the two previous equalities that
%
%
			\begin{equation}
				\label{eq:characDualProjLocal}
				\E^\P\left[\1_F\left(\langle M \rangle_{u\wedge \tau_n} - \langle M \rangle_{t\wedge \tau_n}\right)\right] = \E^\P\left[\int_{t \wedge \tau_n}^{u \wedge \tau_n}N_{r-}^Fd[M]_r\right].
			\end{equation}
The sequences $\left(\langle M \rangle_{u\wedge \tau_n} - \langle M \rangle_{t\wedge \tau_n}\right)_{n \ge 1}$ and $\left(\int_{t \wedge \tau_n}^{u \wedge \tau_n}N_{r-}^Fd[M]_r\right)_{n \ge 1}$ are $\P$-a.s. increasing after a certain rank.
			Letting $n \rightarrow + \infty$ in \eqref{eq:characDualProjLocal}, by monotone convergence, yields
			\begin{equation}
				\label{eq:characDualProj}
				\E^\P\left[\1_F\left(\langle M \rangle_{u} -\langle M \rangle_{t}\right)\right] = \E^\P\left[\int_{t}^{u}N_{r-}^Fd[M]_r\right].
			\end{equation}
			Let then $G \in \shf_t$. Equality
                        \eqref{eq:characDualProj} applied replacing $F$ with
                        $F \cap G$ yields
			\begin{equation}
				\label{eq:condExpStep1}
				\E^\P\left[\1_G\1_F\left(\langle M \rangle_{u} -\langle M \rangle_{t}\right)\right] = E^\P\left[\int_{t}^{u}N_{r-}^{F \cap G}d[M]_r\right].
			\end{equation} 
			For all $r \in [t, u],$ $\E^\P[\1_{F\cap G} | \shf_r] = \1_G \E^\P[\1_F | \shf_r]$, therefore the c\`adl\`ag version of these processes are $\P$-indistinguishable and \eqref{eq:condExpStep1} rewrites
			\begin{equation*}
				\E^\P\left[\1_G\1_F\left(\langle M \rangle_{u} -\langle M \rangle_{t}\right)\right] = \E^\P\left[\1_G\E^\P\left[\int_t^u N_{r-}^Fd[M]_r \middle | \shf_t\right]\right].
\end{equation*}
			Since previous equality holds for any $G \in \shf_t$, we deduce that \eqref{eq:condExpStep2} is verified.
			
                      \item We now prove \eqref{eq:condExpFtXt}. Let $F \in \shf_{t, T}$. Then by Lemma \ref{lemma:measurability1F},
       $\E^\P\left[\1_F \middle | \shf_r\right]$ is $\shf_{t, r}^\P$-measurable for any $r \in [t, u]$.
 Consequently, by the tower property,
 $\E^\P\left[\1_F \middle | \shf_{t, r}\right] = \E^\P[\1_F | \shf_r]$ $\P$-a.s.,
               which implies that $ r \mapsto \E^\P[\1_F | \shf_r], r \in
                       [t,T] $ is also
                        an $\shf_{t, r}^\P$-martingale  and
                     $r \mapsto N_r^F, r \in [t,T]$ is also a càdlag version
                         of 
                         $r \mapsto \E^\P[\1_F | \shf_r]$,
                         as $\shf_{t, r}^\P$-martingale.
                  Consequently    $r \mapsto N_{r-}^F, r \in [t, u],$ is
               $\shf_{t, r}^\P$-adapted.
            Since by Lemma \ref{lemma:quadraticVariation}, $r \mapsto [M]_r - [M]_t$
is also $\shf_{t, r}^\P$-adapted,
then the random variable $\int_t^u N_{r-}^Fd[M]_r$ is $\shf_{t, u}^\P$-measurable.


By Proposition 3.12 in \cite{MartingaleAF} there exists an $\shf_{t, u}$-measurable random variable $Y$
such that $\int_t^u N_{r-}^Fd[M]_r = Y$ $\P$-a.s. Then by \eqref{eq:condExpStep2} and the Markov property given in Proposition \ref{prop:markovPropGen},
 $$
 \E^\P\left[\1_F\left(\langle M \rangle_{u} -\langle M \rangle_{t}\right) | \shf_t\right]
 = \E^\P[Y | \shf_t] = \E^\P[Y | X_t] \quad \P\text{-a.s.}
			 $$
 Since $\sigma(X_t) \subset \shf_t$,
the right-hand side of \eqref{eq:condExpFtXt}, equals the right-hand side of previous equality,
by the 
the tower property of the conditional expectation.
This concludes the proof of   \eqref{eq:condExpFtXt}.
		\end{enumerate}
              \end{proof}

	We are now ready to prove Proposition \ref{prop:FtuMeasurable}.
	\begin{proof}[Proof of Proposition \ref{prop:FtuMeasurable}.]
          Let $0 \le t \le u \le T$. By definition, the compensator of process is adapted, hence $\langle M \rangle_{u} -\langle M \rangle_{t}$ is $\shf_u^\P$-measurable. Since $\shf_{t, u+}^\P = \shf_u^\P \cap \shf_{t, T}^\P$,
          we are going to prove that $\langle M \rangle_{u} -\langle M \rangle_{t}$ is also $\shf_{t, T}^\P$-measurable and
          the result will follow.
          This will be a consequence of the property
          \begin{equation} \label{eq:F}
  \E^\P\left[\1_F\left(\langle M \rangle_{u} -\langle M \rangle_{t}\right)\right]
          = \E^\P\left[\1_F\E^\P\left[\left(\langle M \rangle_{u} -\langle M \rangle_{t} \right) \middle | \shf_{t, T}\right] \right],
        \end{equation}
        for all $F \in \shf$, that we prove below.
        By Lemma \ref{lemma:piSystem}, the set $\Pi := \{F_t \cap F_{t, T} ~:~F_t \in \shf_t, ~F_{t, T} \in \shf_{t, T}\}$ is a $\pi$-system generating $\shf$, hence the previous equality holds if we prove it for 
$F$ of the form
$F = F_t \cap F_{t, T}$ where $F_t \in \shf_t$ and $F_{t, T} \in \shf_{t, T}$.
By Lemma \ref{lm:condExpFtXt} we have
		\begin{equation*}
			\begin{aligned}
				\E^\P\left[\1_F\left(\langle M \rangle_{u} -\langle M \rangle_{t}\right)\right] & = \E^\P\left[\1_{F_t}\E^\P\left[\1_{F_{t, T}}\left(\langle M \rangle_{u} -\langle M \rangle_{t} \right) \middle | \shf_t\right]\right]\\
				& = \E^\P\left[\1_{F_t}\E^\P\left[\1_{F_{t, T}}\left(\langle M \rangle_{u} -\langle M \rangle_{t} \right) \middle | X_t\right]\right]\\
				& = \E^\P\left[\1_{F_t}\E^\P\left[\E^\P\left[\1_{F_{t, T}}\left(\langle M \rangle_{u} -\langle M \rangle_{t} \right) \middle | \shf_{t, T}\right] \middle | X_t\right]\right],
			\end{aligned}
		\end{equation*}
		where we have used the fact that $\sigma(X_t) \subset \shf_{t, T}$ and the tower property of the conditional expectation for the latter equality. Since the random variable $\E^\P\left[\1_{F_{t, T}}\left(\langle M \rangle_{u} -\langle M \rangle_{t} \right) \middle | \shf_{t, T}\right]$ is $\shf_{t, T}$-measurable, we can apply the Markov property \eqref{eq:markovPropGen}
  and we get
		\begin{equation*}
			\begin{aligned}
				\E^\P\left[\1_F\left(\langle M \rangle_{u} -\langle M \rangle_{t}\right)\right] & = \E^\P\left[\1_{F_t}\E^\P\left[\E^\P\left[\1_{F_{t, T}}\left(\langle M \rangle_{u} -\langle M \rangle_{t} \right) \middle | \shf_{t, T}\right] \middle | \shf_t\right]\right] \\
				& = \E^\P\left[\1_{F_t}\E^\P\left[\1_{F_{t, T}}\left(\langle M \rangle_{u} -\langle M \rangle_{t} \right) \middle | \shf_{t, T}\right] \right]\\
				& = \E^\P\left[\1_{F_t}\1_{F_{t, T}}\E^\P\left[\left(\langle M \rangle_{u} -\langle M \rangle_{t} \right) \middle | \shf_{t, T}\right] \right]\\
				& = \E^\P\left[\1_F\E^\P\left[\left(\langle M \rangle_{u} -\langle M \rangle_{t} \right) \middle | \shf_{t, T}\right] \right].
			\end{aligned}
                      \end{equation*}
                      This finally implies \eqref{eq:F}
                and concludes the proof.
	\end{proof}
	We will need the following technical result.
	\begin{lemma}
		\label{lemma:txVersion}
		Let $\gamma : [0, T] \times \Omega \rightarrow \R$ be a progressively measurable process.
                Assume that $\gamma(t, X)$ is $\sigma^\P(X_t)$-measurable
for a.e. $t \in [0, T]$.
                There exists a Borel function $\Gamma \in \shb([0, T] \times \R^d, \R)$ such that $\gamma(t, X) = \Gamma(t, X_t)$ $dt \otimes d\P$-a.e.
	\end{lemma}
	\begin{proof}
		\begin{enumerate}
                \item Assume first that $\E^{\P}\left[\int_0^T |\gamma(t, X)|dt\right] < + \infty$. By Proposition 5.1 in \cite{MimickingItoGeneral}, there exists a measurable function $\Gamma : [0, T] \times \R^d \rightarrow \R$ such that $\E^\P\left[\gamma(t, X)|X_t\right] = \Gamma(t, X_t)$ $dt \otimes d\P$-a.e.
                  Since, for almost all $t$,  $\gamma(t, X)$ is $\sigma^\P(X_t)$-measurable, by Proposition 3.12 in \cite{MartingaleAF} there exists a $\sigma(X_t)$-measurable random variable $Z_t$ such that $\gamma(t, X) = Z_t$ $\P$-a.s. and we have
			\begin{equation*}
				\Gamma(t, X_t) = \E^\P[\gamma(t, X) | X_t] = \E^\P[Z_t | X_t] = Z_t = \gamma(t, X) \quad \P\text{-a.s.}
			\end{equation*}
			Therefore,
			\begin{equation*}
				\begin{aligned}
					\E^\P\left[\left|\int_0^T(\Gamma(r, X_r) - \gamma(r, X))dr\right|\right] \le \int_0^T \E^\P\left[|\Gamma(r, X_r) - \gamma(r, X)|\right]dr = 0,
				\end{aligned}
			\end{equation*}
and we conclude that $\gamma(t, X) = \Gamma(t, X_t)$ $dt \otimes d\P$-a.e.
			
			\item Now for a generic progressively measurable process $\gamma$, set $\gamma_n := (-n) \vee \gamma \wedge n$. Then $\E^{\P}\left[\int_0^T |\gamma_n(t, X)|dt\right] \le n T < + \infty$ and by item 1. there exists $\Gamma_n \in \shb([0, T] \times \R^d, \R)$ such that $\gamma_n(t, X) = \Gamma_n(t, X_t)$ $dt \otimes d\P$-a.e. We set $\Gamma := \underset{n \rightarrow + \infty}{\limsup}~\Gamma_n$. Then $\Gamma \in \shb([0, T] \times \R^d, \R)$ and
			$$
			\gamma(t, X) = \lim_{n \rightarrow + \infty} \gamma_n(t, X) = \limsup_{n \rightarrow + \infty} \gamma_n(t, X) = \limsup_{n \rightarrow + \infty} \Gamma_n(t, X_t) = \Gamma(t, X_t) \quad dt \otimes d\P\text{-a.e.}
			$$
		\end{enumerate}
              \end{proof}

	\begin{prop}
		\label{prop:existenceGamma}
		Let $M \in \shh_{loc}^2$ such that $M_u - M_t$ is $\shf_{t, u}$-measurable for all
                $0 \le t \le u \le T$. Assume that $d\langle M \rangle \ll dt$. Then there exists a Borel function $\Gamma : [0, T] \times \R^d \rightarrow \R$ such that
		\begin{equation}
			\label{eq:angleBracketGamma}
			\langle M \rangle = \int_0^\cdot \Gamma(r, X_r)dr \quad \P\text{-a.s.}
		\end{equation}
	\end{prop}
        
        \begin{proof}
          
          By Proposition 3.5, Chapter I in \cite{JacodShiryaev}, there exists a non-negative progressively measurable process $A$ such that, $\P$-a.s., $\langle M\rangle = \int_0^\cdot A_rdr$.  In view of Lemma \ref{lemma:txVersion} we are going to prove that $A_t$ is $\sigma^\P(X_t)$-measurable
for almost all $t \in [0, T]$.


We fix $t \in [0, T[$. For all $n \in \N^*$, it holds that
		\begin{equation}
			\label{eq:approxIntegral}
			n\left(\langle M \rangle_{t + \frac{1}{n}} - \langle M \rangle_t\right) = n \int_t^{t + \frac{1}{n}}A_rdr \quad \P\text{-a.s.}
                      \end{equation}
By Lemma \ref{lemma:measurablility} applied with $\shg = \shf_t$,
      there is a $\shb([0,T]) \otimes \shf_t$-measurable version $\eta^t$
                      of the 
                      process   $(\E^\P[A_r | \shf_t])_{r \in [0, T]}$.
Let us fix $\omega$ outside a suitable $\P$-null set.
                      By Lebesgue differentiation theorem, for almost all $r \in [0,T]$, we get
                    \begin{equation*}
          \eta^t(s) = 
          \lim_{n \rightarrow + \infty} n\int_s^{s + \frac{1}{n}} \eta^t(r)dr.
           \end{equation*}
Consequently we can write, for all $s \in [0,T]$,
    \begin{equation} \label{eq:E1}
          \eta^t(s) = 
          \liminf_{n \rightarrow + \infty} n\int_s^{s + \frac{1}{n}} \eta^t(r)dr.
        \end{equation}
        Recall that $A$ is non-negative and $(\shf_t)$-adapted.
        Then by Fubini's theorem for the conditional expectation,
$\P$-a.s.  and for all $t \in [0,T]$,    \eqref{eq:E1} implies
        \begin{eqnarray*}
          A_t &=& \eta^t (t) = \liminf_{n \rightarrow + \infty} n\int_t^{t + \frac{1}{n}} \eta^t(r)dr 
 = \liminf_{n \rightarrow + \infty}  n\int_t^{t + \frac{1}{n}} \E[A_r \vert \shf_t] dr \\
     &=&             \liminf_{n \rightarrow + \infty} \E\left(n \int_t^{t + \frac{1}{n}} A_r dr
                  \vert \shf_t\right) 
          =
        \liminf_{n \rightarrow + \infty} \E\left(n \int_t^{t + \frac{1}{n}} A_r dr
                  \vert X_t\right).           
        \end{eqnarray*}
        In fact, last equality has used the fact that, by Proposition \ref{prop:FtuMeasurable},
        $\int_t^{t + \frac{1}{n}}A_rdr$ is $\shf_{t, t + \frac{2}{n}}$-measurable and  Markov property stated
        in Proposition \ref{lemma:markovPropGeneral}.

    
        This finally proves that $A_t$ is $\sigma(X_t)$-measurable
        so that Lemma \ref{lemma:txVersion} can be applied.
        This allows to conclude the proof.
          \end{proof}
        
	Below we state Proposition 4.16 in \cite{MartingaleAF}.
	\begin{prop}
		\label{prop:ortho}
		Let $\shh^{2, dt} := \{M \in \shh^2_0~:~d\langle M \rangle \ll dt\}$ and $\shh^{2, \perp dt} := \{M \in \shh^2_0~:~d\langle M \rangle \perp dt\}$. $\shh^{2, dt}$ and $\shh^{2, \perp dt}$ are orthogonal sub-Hilbert spaces of $\shh^2_0$, $\shh^2_0 = \shh^{2, dt}\oplus \shh^{2, \perp dt}$, and any element $M$ of $\shh^{2, dt}_{loc}$ is strongly orthogonal to any element $N$ of $\shh^{2, \perp dt}_{loc}$, i.e. $\langle M, N \rangle = 0$.
	\end{prop}
Finally we proceed to the proof of Proposition \ref{prop:exitenceGamma}.



\begin{proof} [Proof of Proposition \ref{prop:exitenceGamma}]
		Given a square integrable martingale $\shn$ we denote by $\shn^{dt}$ its orthogonal projection on $\shh^{2, dt}$.
		
		Let $(\tau_n)_{n \ge 1}$ be a sequence of stopping times such that $M^{\tau_n}, N^{\tau_n} \in \shh^2_0(\P)$. Since $\langle N^{\tau_n} \rangle = \langle N\rangle^{\tau_n}$, it is immediate that $d\langle N^{\tau_n} \rangle \ll dt$. Hence $N^{\tau_n} \in \shh^{2, dt}_0$ and by Proposition \ref{prop:ortho},
		\begin{equation*}
			\label{eq:posNegId}
			\begin{aligned}
		\left\langle M^{\tau_n}, N^{\tau_n}\right\rangle & = \left\langle (M^{\tau_n})^{dt}, N^{\tau_n}\right\rangle\\
				& = \frac{1}{4}\left\langle (M^{\tau_n})^{dt} + N^{\tau_n} \right\rangle - \frac{1}{4}\left\langle (M^{\tau_n})^{dt} - N^{\tau_n}\right\rangle.
			\end{aligned}
		\end{equation*}
		Consequently
		$$
		Pos\left(\left\langle M^{\tau_n}, N^{\tau_n}\right\rangle\right) = \frac{1}{4}\left\langle (M^{\tau_n})^{dt} + N^{\tau_n} \right\rangle
		$$
		and
		$$
		Neg\left(\left\langle M^{\tau_n}, N^{\tau_n}\right\rangle\right) = \frac{1}{4}\left\langle (M^{\tau_n})^{dt} - N^{\tau_n}\right\rangle.
		$$
		Recall that $N^{\tau_n} \in \shh^{2, dt}_0$ and that $\shh^{2, dt}_0$ is a linear space.
                Consequently, by additivity,
                we have
		$dPos\left(\left\langle M^{\tau_n}, N^{\tau_n}\right\rangle\right) \ll dt$ and $dNeg\left(\left\langle M^{\tau_n}, N^{\tau_n}\right\rangle\right) \ll dt$. Moreover, it holds
		\begin{equation}
			\label{eq:decompPos}
			Pos\left(\left\langle M, N\right\rangle\right) = \sum_{n \ge 1}Pos\left(\left\langle M^{\tau_n}, N^{\tau_n}\right\rangle\right)\1_{]\tau_n, \tau_{n + 1}]}
		\end{equation}
		and
		\begin{equation}
			\label{eq:decompNeg}
			Neg\left(\left\langle M, N\right\rangle\right) = \sum_{n \ge 1}Neg\left(\left\langle M^{\tau_n}, N^{\tau_n}\right\rangle\right)\1_{]\tau_n, \tau_{n + 1}]}.
		\end{equation}
		It follows from \eqref{eq:decompPos} and \eqref{eq:decompNeg} and the preceding that $dPos\left(\left\langle M, N\right\rangle\right) \ll dt$ as well as $dNeg\left(\left\langle M, N\right\rangle\right) \ll dt$. Proposition \ref{prop:existenceGamma}, applied with $M$
replaced by
$ \frac{1}{2}(M^{dt} + N)$ and $M = \frac{1}{2}(M^{dt} - N)$,
ensures the existence of two functions $\Gamma_+$ and $\Gamma_-$ in $\B([0, T] \times \R^d)$ such that
		$$
		Pos\left(\left\langle M, N\right\rangle\right) = \int_0^\cdot \Gamma_+(r, X_r)dr \quad \text{and} \quad Neg\left(\langle M, N\rangle\right) = \int_0^\cdot \Gamma_-(r, X_r)dr \quad \P\text{-a.s.}
		$$
		Setting $\Gamma = \Gamma_+ - \Gamma_-$, by additivity, we conclude the proof.
              \end{proof}

This finally concludes the proof of Proposition  \ref{prop:exitenceGamma}.

\section{Examples of applications}

\label{sec:applications}

In this section we will provide examples of probability measures
$\P$ which are solutions of martingale problem with respect to
$(\shd, a,\mu)$, for some integro-PDE operators $a$ and some algebra $\shd$.
In this case $(\P, \shd)$ will fulfill the Ideal Property \ref{cond:idealCondition}
since it is Markovian. 
$v$ will still denote an intrinsic value function introduced in Definition \ref{rmk:Function_v}.
 
 \begin{remark} \label{rmk:vtilde}
 $\Gamma^{v,c}: \shd \rightarrow L^0$ is the map defined
in the sense of
Proposition \ref{prop:GammaContinuous}; the same notation
will refer to its extension to the closure on $\scrd = C^{0,1}$ in
Proposition \ref{prop:extensionGammaClosure}.
\end{remark}
In the examples below, the Markov property, stated in Hypothesis
\ref{hyp:markovProp} for a given probability $\P$, is verified indirectly,
in particular proving that $\P$ is {\it Regularly Markovian},
see Definition \ref{def:regularlyMarkovian}.
In fact a Regularly Markovian probability 
fulfills the Markov property
mentioned above, see Proposition \ref{lemma:markovPropGeneral}.

\subsection{Markovian jump diffusions}

We focus in this section on the case when $\P$ is the law of a  Markovian diffusion with jumps, namely
when $\P$
fulfills a martingale problem in the framework of  Hypothesis \ref{hyp:existencePsxJumpDiffusion} item 1.
with respect to some truncation function $k$.

        \begin{remark} \label{rmk:jumps}
  We list below some possible assumptions on the coefficients
  for which there is a solution $\P$  which is Regularly Markovian
  fulfilling the martingale problem described in 
   Hypothesis \ref{hyp:existencePsxJumpDiffusion} item 1.,
  with respect to some truncation function $k$.
  	\begin{itemize}
        \item $b$
          has linear growth uniformly in $t$, $\Sigma$ is bounded continuous and non-degenerate.
          Moreover there exists a measure $L_*$ on $\shb(\R^d \backslash \{0\})$ such that $\int_{\R^d} \left(1 \wedge |q|^2\right) L_*(dq) < + \infty$ and for all $(t, x) \in [0, T] \times \R^d$, $L_*(\cdot) - L(t, x, \cdot)$ is a non-negative measure. See Theorem 5.2 in \cite{KomatsuJumpsDiffusion}.

		\item There is no diffusion component, $b$ is bounded continuous, $(t, x) \mapsto \int \frac{|q|^2}{1 + |q|^2}\vphi(q)L(t, x, dq)$ is
                  continuous for all $\vphi \in \shc_b(\R^d)$ and $\underset{(t, x) \in [0, T] \times \R^d}{\sup}\int_{\R^d}(1 \wedge |q|^2)L(t, x, dq) < + \infty$,
                  see Theorem 2.2 in \cite{StroockDiffusionLevy} for the existence and Theorem 3 in \cite{KomatsuPureJumps} for the uniqueness. This typically includes the case of $\alpha$-stable Lévy processes, for instance Cauchy processes. 
	\end{itemize}
\end{remark}

\begin{prop}
	\label{prop:markovianDriftJumps}
	Let $\P$ in the framework of Hypothesis \ref{hyp:existencePsxJumpDiffusion} item 1.
        with respect to some truncation function $k$,
     which fulfills
Hypothesis \ref{hyp:markovProp}. 
        Set $\mu := \shl^\P(X_0)$.
Let $\nu$ be defined in \eqref{eq:nudx}   and 
$v$ an intrinsic value function, see Definition \ref{definitionV}.

Then the exponential twist $\Q$ given by (\ref{eq:Qstar})
is solution of a martingale problem $(\shd, a^{\Q}, \nu)$
and for all $\phi \in \shd$
	\begin{equation}
		\label{eq:newGeneratorJumpDiffusion}
		\begin{aligned}
                  	a^{\Q}(\phi)(t, x) & = a(\phi)(t, x) +  \left\langle \nabla_x \phi(t, x), \frac{\Gamma(v)(t, x)}{v(t, x)}\right\rangle \\
			& + \int_{\R^d}\left(\frac{v(t, x + q)}{v(t, x)} - 1\right)(\phi(t, x + q) - \phi(t, x))L(t, x, dq),
		\end{aligned}
              \end{equation}

	which rewrites

	\begin{equation}
		\label{eq:newGeneratorJumpDiffusion2}
		\begin{aligned}
             a^{\Q}(\phi)(t, x)
             & = \partial_t\phi(t, x) + \left\langle \nabla_x \phi(t, x), b(t, x) + \frac{\Gamma(v)(t, x)}{v(t, x)} + \int_{\R^d}k(q)\left(\frac{v(t, x + q)}{v(t, x)} - 1\right)L(t, x, dq)\right\rangle\\
			& + \frac{1}{2}Tr[\nabla_x^2\phi(t, x)\sigma\sigma^\top(t, x)]\\
			& + \int_{\R^d}\left(\phi(t, x + q) - \phi(t, x) - \langle\nabla_x\phi(t, x), k(q)\rangle\right)\frac{v(t, x + q)}{v(t, x)}L(t, x, dq),
		\end{aligned}
	\end{equation}
	where
        \begin{equation} \label{eq:GammavJumps}
          \Gamma(v)(s, x) := \left[\Gamma^{v,c}(id_i)(s, x)
          \right]_{1 \le i \le d},
	\end{equation}
       and $\Gamma^{v,c}$ is the map mentioned in Remark \ref{rmk:vtilde}. 
\end{prop}
\begin{remark} \label{rmk:intfv}
  \begin{enumerate}
  \item We insist on the fact that $\P$ of previous proposition
  fulfills  Hypothesis  \ref{hyp:compensator}. 
  \item  Note that by Lemma \ref{lemma:intCompensator}, taking into
    account 
Hypothesis \ref{hyp:compensator}, 
the function
\begin{equation} \label{eq:intfv}
  (t,x) \mapsto \int_{\R^d}k(q)\left(\frac{v(t, x + q)}{v(t, x)} - 1\right)L(t, x, dq),
\end{equation}
  is a well-defined
element of $L^0(\Q) (= L^0(\P))$. 
\end{enumerate}
\end{remark}
 \begin{proof}[Proof of Proposition \ref{prop:markovianDriftJumps}]
  Since $\shd$ is an algebra, Theorem \ref{thm:IdealCond} states that $(\P,\shd)$ fulfills the
   Ideal Property \ref{cond:idealCondition}. 

   The first equality \eqref{eq:newGeneratorJumpDiffusion} is then a direct application of Theorem \ref{th:markovDrift},
    taking into account Proposition \ref{prop:idGamma}. The second equality \eqref{eq:newGeneratorJumpDiffusion2} follows by noticing that
    \begin{equation}
    	\label{eq:vPhi}
	\begin{aligned}
    		& \int_{\R^d}\left(\frac{v(t, x + q)}{v(t, x)} - 1\right)(\phi(t, x + q) - \phi(t, x))L(t, x, dq) \\
          = & \int_{\R^d}\frac{v(t, x + q)}{v(t, x)}(\phi(t, x + q) - \phi(t, x) - \langle\nabla_x\phi(t, x), k(q)\rangle) L(t,x,dq) \\
          -& \int_{\R^d}(\phi(t, x + q) - \phi(t, x) - \langle\nabla_x\phi(t, x), k(q)\rangle) L(t,x,dq)  \\
    		+ & \left\langle \nabla_x \phi(t, x), \int_{\R^d}k(q)\left(\frac{v(t, x + q)}{v(t, x)} - 1\right)L(t, x, dq)\right\rangle,
    	\end{aligned}
      \end{equation}
      	and by injecting \eqref{eq:vPhi} in \eqref{eq:newGeneratorJumpDiffusion} taking into account \eqref{eq:generatorJumps}.
\end{proof}

\subsection{The case of Markovian diffusions}
\label{sec:markovDiff}

We consider here the particular case of Brownian diffusions, i.e.
when $\P$ verifies Hypothesis \ref{hyp:existencePsxJumpDiffusion} item 2.

\begin{remark} \label{rmk:NoAss}
  We emphasize that we do not make any assumption on the coefficients $\sigma, b$ of the martingale problem. In fact, we do not even require local boundedness of these coefficients.
  All the results of this paper are based on the
simple Markov property of
the probability measure $\P$
without any restriction
  on the generator $a$ of the underlying martingale problem.
\end{remark}

\begin{remark} \label{rmk:MarkovCC}
  Hypothesis \ref{hyp:markovProp}, via Definition
  \ref{def:regularlyMarkovian} and Proposition \ref{prop:markovGenCond},
    is verified for instance in the following cases.

	\begin{itemize}
		\item $\sigma, b$ have linear growth and $\sigma$ is continuous and non-degenerate, see e.g. \cite{stroock} Corollary 7.1.7 and Theorem 10.2.2.
		\item $d=1$ and $\sigma$ is lower bounded by a positive constant on each compact set, see \cite{stroock}, Exercise 7.3.3.
		\item $d =2$, $\sigma\sigma^\top$ is non-degenerate and $\sigma$  and $b$ are time-homogeneous and bounded, see \cite{stroock}, Exercise 7.3.4.
		\item $\sigma, b$ are Lipschitz with linear growth (with respect to the space variable, independently in time).
		\item $\sigma, b$ are bounded and continuous, see Chapter 12 in \cite{stroock} and the Markov selection therein.
	\end{itemize}
      \end{remark}

      \begin{coro}
	\label{coro:markovianDrift}
	Let $\P$ be solution of the martingale problem mentioned
        in Hypothesis \ref{hyp:existencePsxJumpDiffusion} item 2. 
        and Hypothesis \ref{hyp:markovProp}.
        Set $\mu := \shl^\P(X_0)$.
	Let $\nu$ be defined in \eqref{eq:nudx} and
	$v$ in Definition \ref{definitionV}.
        Then the exponential twist
        $\Q$  given by (\ref{eq:Qstar})
	is solution of a martingale problem
	$(\shd, a^{\Q}, \nu),$ where,
	for all $\phi \in \shd$
	\begin{equation}
		\label{eq:newGeneratorDiffusion}
		a^{\Q}(\phi)(t, x) = \partial_t \phi(t, x) + \left\langle \nabla_x \phi(t, x), b(t, x) + \frac{\Gamma(v)(t, x)}{v(t, x)}\right\rangle + \frac{1}{2}Tr[\sigma\sigma^\top(t, x)\nabla_x^2\phi(t, x)],
	\end{equation}
	where
	\begin{equation} \label{eq:Gammav}
		\Gamma(v)(s, x) := \left[\Gamma^{v,c}(id_i)(s, x)\right]_{1 \le i \le d},
	\end{equation}
 where $\Gamma^{v,c}$ is again the map mentioned in Remark \ref{rmk:vtilde} 2.
\end{coro}
\begin{proof}
  Since $\P$ verifies Hypothesis \ref{hyp:compensator}
trivially taking $L \equiv 0$,
  the result follows directly from Proposition \ref{prop:markovianDriftJumps}.
\end{proof}

\begin{remark} \label{rmk:GenGrad}
  When the function $v$ defined in Definition
  \ref{definitionV} is an element of $\shc^{0, 1},$  then Corollary
  \ref{cor:P1PDE} implies that 
  $\Gamma(v) = \sigma\sigma^\top \nabla_x v. $
    Indeed, by item 2. of Proposition \ref{prop:MarkovMart}, $a^\P(v) = fv$.
\end{remark}

We state now some consequences which are used in the companion
paper \cite{BOROptimi2023}, when the coefficients $b, \sigma$ have linear growth.
To avoid more technical conditions we will suppose the initial condition
to be deterministic, i.e. $\mu = \delta_x$, for  some $x \in \R^d$.
\begin{hyp}
	\label{hyp:coefDiffusion}
	\begin{enumerate}
		\item  The coefficients $b, \sigma$ in \eqref{eq:generatorDiffBrownian} satisfy
		\begin{equation} \label{item:linGrowth} 
		|b(t, x)| + \|\sigma(t, x)\| \le C(1 + |x|),
		\end{equation}
		for some constant $C > 0$.
		
		\item \label{item:unifEllip} $\sigma$ is uniformly elliptic in the sense that, for all $(t, x) \in [0, T] \times \R^d$, $\xi \in \R^d$, $\xi^\top \sigma\sigma^\top(t, x)\xi \ge c_\sigma |\xi|^2,$ for some constant $c_\sigma > 0$.
	\end{enumerate}
\end{hyp}
\begin{lemma} 
	\label{lemma:integrabilityGamma}
	Assume Hypothesis \ref{hyp:coefDiffusion}. Let $\P$ be a solution of the martingale problem $(\shd, a, \delta_x)$ for some $x \in \R^d$ and operator $a$ defined by \eqref{eq:generatorDiffBrownian}.
	Let $v$ be the function defined in Definition \ref{definitionV}.
	and $\Gamma(v)$ defined in \eqref{eq:Gammav}. 
	Then for all $1 < p < 2$,
	$$
	\E^{\Q}\left[\int_0^T \left|\frac{\Gamma(v)(r, X_r)}{v(r, X_r)}\right|^pdr\right] < + \infty.
	$$
\end{lemma}
      
\begin{proof}
  Since Hypothesis \ref{hyp:coefDiffusion} holds, $(\P, \shc_b^{1, 2})$
  verifies
  Hypothesis \ref{hyp:markovProp},
  see Remark \ref{rmk:MarkovCC}. Then by Corollary \ref{coro:markovianDrift}, under $\Q$ the canonical process decomposes into
	\begin{equation}
		\label{eq:decompCoro}
		X_t = x + \int_0^t b(r, X_r)dr + \int_0^t \frac{\Gamma(v)(r, X_r)}{v(r, X_r)}dr + M^{\Q}_t,
	\end{equation}
	where $M^{\Q}$ is a $\Q$-local martingale such that $\langle M^{\Q}\rangle_t = \int_0^t \sigma\sigma^\top(r, X_r)dr$. This decomposition is a direct consequence of Proposition 5.4.6 in \cite{karatshreve}, noticing that the martingale problem verified by $\P$ extends to $\shd = C^{1,2}([0,T], \R^d)$.
        On the other hand, since $H(\Q| \P) < + \infty$, Theorem 2.1 in \cite{GirsanovEntropy} gives the existence of a progressively measurable process $\alpha$ such that
	\begin{equation}
		\label{eq:entropBound}
		\E^{\Q}\left[\int_0^T|\sigma^\top(r, X_r)\alpha_r|^2dr\right] < + \infty,
	\end{equation}
	and under $\Q$ the canonical process has decomposition
	\begin{equation}
		\label{eq:decompGirsa}
		X_t = x + \int_0^t b(r, X_r)dr + \int_0^t \sigma\sigma^\top(r, X_r)\alpha_rdr + \tilde M_t,
	\end{equation}
	where the $\Q$-local martingale $\tilde M$ verifies $\langle \tilde M\rangle_t = \int_0^t \sigma\sigma^\top(r, X_r)dr$. Identifying the bounded variation and the martingale components between \eqref{eq:decompCoro} and \eqref{eq:decompGirsa}, we get $\tilde M = M^{\Q}$ and
\begin{equation} \label{eq:leonard}
    \sigma\sigma^\top(\cdot, X_\cdot)\alpha_\cdot = \Gamma(v)(\cdot, X_\cdot)/v(\cdot, X_\cdot), \ dt\otimes d\Q{\rm-a.e}.
\end{equation}
Besides, since $\|d\Q/d\P\|_\infty < + \infty$, the linear growth
of $\sigma$ and classical moments estimates under $\P$ yield
	\begin{equation}
		\label{eq:sigmaMoment}
		\E^{\Q}\left[\int_0^T \|\sigma(r, X_r)\|^qdr\right] < + \infty,
	\end{equation}
for all $q \ge 1$. We then fix $1 < p < 2$. By H\"older's inequality applied on the measure space $([0, T] \times \Omega, {\rm Bor}([0, T])\otimes \F, dt\otimes d\Q)$, it holds that
	\begin{equation*}
		\begin{aligned}
			\E^{\Q}\left[\int_0^T |\sigma\sigma^\top(r, X_r)\alpha_r|^pdr\right] & \le 	\E^{\Q}\left[\int_0^T \|\sigma(r, X_r)\|^p|\sigma^\top(r, X_r)\alpha_r|^pdr\right]\\
			& \le \left(\E^{\Q}\left[\int_0^T \|\sigma(r, X_r)\|^{2p/(2 - p)}dr\right]\right)^{1 - p/2}\left(\E^{\Q}\left[\int_0^T |\sigma^\top(r, X_r)\alpha_r|^2dr\right]\right)^{p/2}.
		\end{aligned}
              \end{equation*}
      By   \eqref{eq:leonard}  \eqref{eq:entropBound} and \eqref{eq:sigmaMoment}, we get 
	$$
	\E^{\Q}\left[\int_0^T \left|\frac{\Gamma(v)(r, X_r)}{v(r, X_r)}\right|^pdr\right] = \E^{\Q}\left[\int_0^T |\sigma\sigma^\top(r, X_r)\alpha_r|^pdr\right] < + \infty.
	$$
\end{proof}
The corollary below  constitutes a key tool in \cite{BOROptimi2023},
see Proposition 4.6.

\begin{coro}
	\label{coro:markovianDriftSDE}
	Let $X$ be a solution in law of the SDE
       	\begin{equation*}
		\left\{
		\begin{aligned}
			& dX_t = b(t, X_t)dt + \sigma(t, X_t)dW_t\\
			& X_0 = x,
		\end{aligned}
		\right.
	\end{equation*}
	where $b, \sigma$ verify Hypothesis \ref{hyp:coefDiffusion}. Then
	there exists a function $\lambda \in \B([0, T] \times \R^d, \R^d)$ such that
       the exponential twist $\Q$ given by (\ref{eq:Qstar}) 
	is the law of a weak solution of the SDE
	\begin{equation*}
		\left\{
                  \begin{aligned}
                  	& dX_t = \left(b(t, X_t) + \lambda(t, X_t)\right)dt + \sigma(t, X_t)dW_t\\
			& X_0 = x.
		\end{aligned}
		\right.
	\end{equation*}
Moreover, $(t,\omega) \mapsto \lambda(t,X_t(\omega)) \in L^p([0, T] \times \Omega, dt \otimes d\Q)$ for all $1 \le p < 2$.
\end{coro}
\begin{proof}
  Since Hypothesis \ref{hyp:coefDiffusion} holds, $\P$
  verifies  Hypothesis \ref{hyp:markovProp},
 see Remark
  \ref{rmk:MarkovCC}, so that we can apply
  Corollary \ref{coro:markovianDrift} and Lemma \ref{lemma:integrabilityGamma}.
  The result follows from the equivalence between weak solution of SDEs and solution of martingale problems associated to $(\shd, a, \delta_x)$, where $a$ is given by \eqref{eq:generatorDiffBrownian}, see e.g. Proposition 5.4.6 in \cite{karatshreve}.

\end{proof}

\subsection{SDEs with distributional drift}
\label{sec:distriContinuous}

We apply in this section our result to a more irregular framework, where the reference probability measure $\P$ is solution of a martingale problem with
parabolic generator $a(\phi) = \partial_t \phi + \langle \nabla_x \phi, b \rangle + \frac{1}{2}\Delta \phi$, where the drift $b = (b^1, \ldots,b^d)$  is a
(vector valued Schwartz) distribution. We will use throughout this section the formalism and some of the results from \cite{IssoglioRussoDistDriftBesov}.
For $\gamma \in \R$ we denote $\shc^\gamma(\R^d)$ the Besov space $\B^{\gamma}_{\infty, \infty}$. For details on Besov spaces we refer to Section 2.7 in \cite{BesovSpaces}. In particular, for any $\phi \in \shc^{\alpha}(\R^d)$, $\psi \in \shc^{-\beta}(\R^d)$ for $\alpha, \beta > 0$ such that $\alpha - \beta > 0$, one can define the
pointwise product $\phi\psi \in \shc^{-\beta}(\R^d)$.
We also define $\shc^{\gamma +}(\R^d) := \underset{\alpha > \gamma}{\bigcup}\shc^\alpha(\R^d).$
$\shc_c^\gamma(\R^d)$ will denote the set of elements of $\shc^\gamma(\R^d)$ with compact support. Finally we denote $\bar \shc_c^{\gamma}(\R^d)$ the space
$$
\bar \shc_c^{\gamma}(\R^d) := \{\phi \in \shc^\gamma(\R^d)~:~\exists (\phi_n) \subset \shc_c^\gamma(\R^d) ~\textit{such that}~(\phi_n) \rightarrow \phi ~\textit{in}~\shc^{\gamma}(\R^d)\}
$$
and we define the spaces $\bar \shc_c^{\gamma +}(\R^d)$
as $\shc^{\gamma +}(\R^d).$

Let $0 < \beta < \frac{1}{2}$. According
to Theorem 4.5 in \cite{IssoglioRussoDistDriftBesov},
given a Borel probability law on $\R^d$,
there is a unique probability measure $\P$ being solution
to the martingale problem (with distributional
drift)  with respect to
 $(\shd, a,\mu)$,
where
	 	\begin{equation}
			\label{eq:defDaDistri}
			\begin{aligned}
				\shd := \Big\{& \phi \in C\left([0, T], \shc^{(1 + \beta) +}(\R^d)\right) ~:~\exists \vphi \in C\left([0, T], \bar \shc_c^{0 +}(\R^d)\right)~\text{such that}\\
& \phi~\text{is a weak solution of}~a(\phi) = \vphi~\text{and}~\phi(T, .) \in \bar \shc_c^{(1 + \beta)+}(\R^d)\Big\}
            \end{aligned}
		\end{equation}
		and $a(\phi) = \partial_t \phi + \langle \nabla_x \phi, b \rangle + \frac{1}{2}\Delta \phi$ for a drift $b \in C\left([0, T], \shc^{(-\beta) +}(\R^d)\right)$.                
   We remark that $ \langle \nabla_x \phi, b\rangle:= \sum_i \partial_{x_i} \phi  b^i$
and the products $\partial_{x_i} \phi  b^i$ are pointwise products.
\begin{remark} \label{rmk:RegMark}
	\begin{enumerate}
        \item The aforementioned probability measure $\P$ is
          Markovian since it is the Zvonkin transform of a probability measure fulfilling a martingale problem of the same type as the one in the first bullet point in Remark \ref{rmk:MarkovCC}, see Theorem 3.9 in \cite{IssoglioRussoDistDriftBesov}.
		                  
		\item The martingale problem in \cite{IssoglioRussoDistDriftBesov} is stated on the canonical space of the continuous functions on $C([0,T], \R^d)$ instead on $D([0,T], \R^d)$. However, using similar arguments as in the discussion following Remark \ref{rmk:MarkovCC} at the level of the Zvonkin transformed process, one can show that the jump measure is necessarily zero, whenever the martingale problem is formulated in the space of c\`adl\`ag functions.
      \item The canonical process $X$ under the probability measure $\P$
        is a weakly finite quadratic variation process.
                \end{enumerate}
\end{remark}

We are now ready to characterize the solution $\Q$ to Problem
\eqref{eq:variationalFormulation} in this framework. 
\begin{prop}
	\label{prop:markovianDriftDistri}
	Let $(\P, \shd)$ introduced above and
set  $\mu := \shl^\P(X_0)$.
	Let $\nu$ be defined in \eqref{eq:nudx}   and
	$v$  defined in Definition \ref{definitionV}.
	Then the exponential twist $\Q$ given by (\ref{eq:Qstar})
	is solution of a martingale problem $(\shd, a^{\Q}, \nu),$ where
        $\shd$ is given by \eqref{eq:defDaDistri} and
	\begin{equation*}
		a^{\Q}(\phi) = a(\phi) + \left\langle \nabla_x \phi, \frac{\Gamma(v)}{v}\right\rangle,
	\end{equation*}
	where $\Gamma(v) := \left(\Gamma^{v,c}(id_i)\right)_{1 \le i \le d}$ is provided by Proposition \ref{prop:extensionGammaClosure}.
\end{prop}
\begin{proof}  
  Since $\shd$ is an algebra,   Theorem \ref{thm:IdealCond}
  together with Remark \ref{rmk:RegMark} imply that $(\P, \shd)$ fulfills the Ideal Property \ref{cond:idealCondition}. Moreover under $\P$ the canonical process $X$ is a continuous weak Dirichlet process by Proposition 5.11 in \cite{IssoglioRussoDistDriftBesov} applied with $f = id$. In particular, $\P$ verifies Hypothesis \ref{hyp:compensator}
  with $L = 0$.
   Since $\shd$ is dense in $\shc^{0, 1}$ by Lemma 5.7 in \cite{IssoglioRussoDistDriftBesov},
we can apply Proposition \ref{prop:idGamma}
which says that $\scrd = C^{0,1}$ is a closure of $\shd$,
$$ \Gamma^v(\phi) = \langle \nabla_x \phi, \Gamma^{v,c}(id)  \rangle, \  \forall \phi \in \scrd $$ 
and $\Gamma^{v,c}(id)$  is provided by  Proposition \ref{prop:extensionGammaClosure}.
The result then follows from Theorem \ref{th:markovDrift}.
\end{proof}

\appendix
\section*{Appendices}

\section{Measurability and generalized conditional expectation
}
\setcounter{equation}{0}
\renewcommand\theequation{A.\arabic{equation}}

	We need the following extension of the conditional expectation to the case of non-negative, not necessarily integrable random variables. We refer to Remark 39, Chapter I in \cite{DellacherieMeyer_I_IV}.

\begin{prop}
	\label{prop:generalizedCondExp}
	(Generalized conditional expectation). Let $(\Omega, \F, \P)$ be a probability space. Let $Y$ be a \textbf{non-negative} random variable on $(\Omega, \F)$. Let $\shg \subset \F$ be a sub-$\sigma$-algebra of $\F$. There exists a unique
      ($\P$-a.s.)  non-negative $\shg$-measurable
        random variable with values in $[0, + \infty]$, denoted $\E[Y | \shg]$, such that for $\E[\1_A Y] = \E[\1_A \E[Y | \shg]]$ for all $A \in \shg$.
\end{prop}

\medskip
	\begin{prop}
		\label{prop:markovPropGen}
		Let $\P \in \shp(\Omega)$ be a probability measure satisfying Hypothesis \ref{hyp:markovProp} (Markov property). For all $t \in [0, T]$, $F \in \shb(D([t, T], \R^d), [0, + \infty])$, we have
		\begin{equation}
			\label{eq:markovPropGen}
			\E^{\P}[F((X_r)_{r \in [t, T]}) | \shf_t] = \E^{\P}[F((X_r)_{r \in [t, T]}) | X_t] \quad \P\text{-a.s.}
		\end{equation}
	\end{prop}
	\begin{proof}
		Let $t \in [0, T]$, $F \in \shb(D([t, T], \R^d), [0, + \infty])$. Let $n \in \N^*$. Applying \eqref{eq:markovPropDef} to $F = F \wedge n$ yields
		$$
		\E^{\P}[F((X_r)_{r \in [t, T]}) \wedge n | \shf_t] =
                \E^{\P}[F((X_r)_{r \in [t, T]}) \wedge n | X_t] \quad \P\text{-a.s.}
		$$
and letting $n \rightarrow + \infty$ in the previous equality, we get \eqref{eq:markovPropGen} by the monotone convergence theorem for the conditional expectation.
	\end{proof}

        \medskip

We conclude the section  with two lemmas concerning the measurability of the conditional
expectation.

\begin{lemma}
	\label{lemma:measurability1F}
	Let $t \in [0, T]$ and let $F \in \shf_{t, T}$. Then the random variable $\E^\P\left[\1_F \middle | \shf_r\right]$ is $\shf_{t, r}^\P$-measurable for any $r \in [t, T]$.
\end{lemma}
\begin{proof}
	Let $r \in [t, T]$. Let $F \in \shf_{t, T}$. Assume first that $F = F_{t, r} \cap F_{r, T}$ for some $F_{t, r} \in \shf_{t, r}$ and $F_{r, T} \in \shf_{r, T}$. Then since $\1_{F_{t, r}}$ is $\shf_r$-measurable, we have
		$$
		\E^\P\left[\1_F\middle | \shf_r\right] = \1_{F_{t, r}}\E^\P\left[\1_{F_{r, T}} \middle | \shf_r\right] = \1_{F_{t, r}}\E^\P\left[\1_{F_{r, T}} \middle | X_r\right]~\P\text{-a.s.},
		$$
		where we used the Markov property (Hypothesis
                \ref{hyp:markovProp}) to get to the last equality. Since $\1_{F_{t, r}}$ is $\shf_{t, r}$-measurable and $\E^\P\left[\1_{F_{r, T}} \middle | X_r\right]$ is $\sigma(X_r)$-measurable (hence $\shf_{t, r}$-measurable), we get that $\E^\P\left[\1_F\middle | \shf_r\right]$ is $\shf_{t, r}^\P$-measurable. Moreover, the set
		$$
		\Lambda := \left\{F \in \shf_{t, T}~\middle | \E^\P\left[\1_F\middle | \shf_r\right] ~\text{is}~\shf_{t, r}^\P\text{-measurable}\right\}
		$$
		is clearly a $\lambda$-system, see Section 3. of
Chapter 1, in \cite{billingsley}. By Lemma \ref{lemma:piSystem} together with Remark \ref{rmk:piSystem}, it follows from what precedes that $\Lambda$ contains a $\pi$-system generating $\shf_{t, T}$. By the Dynkin monotone class theorem, see e.g. Theorem 3.2, Chapter 1, in \cite{billingsley}, we conclude that $\Lambda = \shf_{t, T}$.
\end{proof}

	\begin{lemma}
		\label{lemma:measurablility}
		Let $(A_t)_{t \in [0, T]}$ be a measurable non-negative process and let $\shg \subset \shf$ be a sub-$\sigma$-field. Then there is
                a ${\rm Bor}([0,T])\otimes \shg$-measurable process
                $(\eta_t)$ such that
\begin{equation} \label{eq:measurability}
  \eta_t = \E[A_t \vert \shg] \ {\rm a.s.}, \  \forall t \in [0,T].
  \end{equation}
	\end{lemma}
	\begin{proof}
          We will make use of  functional monotone class arguments.
          \begin{itemize}
          \item Suppose that $A_t = \sum_i g_i(t) A^i(\omega), t \in [0,T]$,
            where the $g_i$ (resp. the  $A^i$) are Borel real functions on $[0,T]$
            (resp. a $\shf$-measurable r.v.).
            In this case we set
            $\eta(t) := \sum_i g_i(t) \E^\P[A^i \vert \shg]$. 
          \item Suppose that $A$ is bounded. Then the
            result follows by the functional monotone class theorem,
            see e.g. Theorem 21 of Chapter 14-I in \cite{meyer}, with
            $\shc$ is the set of simple functions of previous item,
            $\shh$ is the space of $\shb([0,T]) \otimes \shg$
            functions fulfilling \eqref{eq:measurability}.
       \item For the general case we set 
         $\eta(t) := \liminf_{n\rightarrow +\infty} \eta^n(t),$
       where
          $ \eta^n(t) = \E^\P[A_t \wedge n \vert \shg],$
          for all $t \in [0,T]$.
            \eqref{eq:measurability} can be established via the
                monotone convergence theorem for the conditional expectation.
              \end{itemize}

       \end{proof}

        \section{Markov canonical classes and
     Markov property   }
        
\setcounter{equation}{0}
\renewcommand\theequation{B.\arabic{equation}}

In this section we introduce the notion of Markov canonical class
and we prove that whenever the reference probability $\P$
is associated with a Markov canonical class,
then it is Markovian, i.e. it fulfills
Hypothesis \ref{hyp:markovProp}
In the literature the notion of Markov canonical class often appears as  the suitable tool for describing solutions of
stochastic differential equations in law.

\begin{definition}
\label{def:markovCanonicalClass}
	(Markov canonical class). 
	Let $(\P^{s, x})_{(s, x) \in [0, T] \times \R^d}$ be a set of probability measures on $(\Omega, \F)$
	with corresponding expectation operator maps $(\E^{s, x})_{(s, x) \in [0, T] \times \R^d}$.
	$(\P^{s, x})_{(s, x) \in [0, T] \times \R^d}$ is called a \textit{Markov canonical class} if the following holds.
    \begin{enumerate}
    \item For all $F \in \F$,  $s \in [0, T]$ the map
      $x \mapsto \P^{s, x}(F)$ is Borel.
      \item  For all $(s, x) \in [0, T] \times \R^d$ and for any $s \le t \le u \le T$, $B \in {\rm Bor}(\R^d)$, 
        $\P^{s, x}(X_r = x, 0 \le r \le s) = 1$ and
	\begin{equation}
		\label{eq:markovPropCondProba}
		\P^{s, x}(X_u \in B~|~\F_t) =
                \P^{t, X_t}(X_u \in B) \quad \P^{s, x}\text{-a.s.}
              \end{equation}
              \end{enumerate}
        We will say that $(\P^{s, x})_{(s, x) \in [0, T] \times \R^d}$ is \textit{measurable in time} if $(s, x) \mapsto \P^{s, x}(F)$ is Borel for all $F \in \F$.
\end{definition}

  \begin{definition}
	\label{def:regularlyMarkovian}
	(Regularly Markovian). A probability measure $\P \in \Pma(\Omega)$ is said to be \textit{Regularly Markovian}
        if there exists a  measurable in time Markov canonical class $(\P^{s, x})_{(s, x) \in [0, T] \times \R^d}$ and a probability measure $\mu \in \Pma(\R^d)$ such that $\P = \int_{\R^d}\P^{0, x}\mu(dx)$.
        The mentioned Markov canonical class will be said  associated with $\P$ and
        $\mu$ will be the \textit{initial law}.
      \end{definition}

      \begin{remark}
	\label{rmk:measurabilityMarkovExpect}
        Let $(\P^{s, x})_{(s, x) \in [0, T] \times \R^d}$ be a Markov canonical class measurable in time. Let $Z$ be a random variable such that
        $\E^{s, x}[Z]$ is well-defined for all $(s, x) \in [0, T] \times \R^d$.
Then we have the following.
        \begin{enumerate}
\item Then $(s, x) \mapsto \E^{s, x}[Z]$ is Borel.
   The proof of that fact can be found in \cite{MartingaleAF}, Proposition 3.10.
 \item  For any $s \le t \le u \le T$, $f: \R^d \rightarrow \R$ bounded or non-negative Borel
   function,  we have
\begin{equation}
		\label{eq:markovPropCondProbaExp}
		\E^{s, x}(f(X_u)|~\F_t) = \E^{t, X_t}(f(X_u)) \quad \P^{s, x}\text{-a.s.}
	\end{equation}
Indeed, when 
    $f = \1_A$, $A \in {\rm Bor}(\R^d)$
 this follows by  the definition of a Markov canonical class, see \eqref{eq:markovPropCondProba}.
        Then by pointwise approximation of any positive function $f \in \B_b(\R^d, \R^+)$ by an increasing sequence
        of simple functions and the monotone convergence theorem for the conditional expectation,
        the property is also true for any $f \in \B_b(\R^d, \R^+)$. This extends to any $f \in \shb_b(\R^d, \R)$ by setting $f = f_+ - f_-$.
\end{enumerate}
      \end{remark}

The objective of the rest of the  section is to prove that a reference probability $\P$,
which is Regularly Markovian, verifies the Markov
property.

\begin{prop}
	\label{prop:PsxSatisfyMarkov}
	Let $(\P^{s, x})_{(s, x) \in [0, T] \times \R^d}$ be a Markov canonical class in the sense of Definition \ref{def:markovCanonicalClass}. Then for all $0 \le s \le t \le T$, $F \in \B_b(D([t, T], \R^d), \R)$ it holds that
	$$
	\E^{s, x}\left[F((X_u)_{u \in [t, T]})~|~\F_t\right] =
        \E^{t, X_t}[F((X_u)_{u \in [t, T]})].
	$$
\end{prop}
We first prove a weaker version of this proposition in order to apply a functional version of the monotone class
theorem to prove Proposition \ref{prop:PsxSatisfyMarkov}.

\begin{lemma}
  \label{lemma:inductionMarkov}
  Let $n \ge 1$, $t \le t_1 \le \ldots \le t_n \le T$,
  $f \in \B_b((\R^d)^n, \R), (s,x) \in [0,T] \times \R^d$.
  We have
\begin{equation}\label{eq:inductionMarkov}
  \E^{s, x}[f(X_{t_1}, \ldots, X_{t_n})~|~\F_t] = \E^{t, X_t}\left[f(X_{t_1}, \ldots, X_{t_n})\right] \ \P^{s, x}{\text -a.s.}
  \end{equation}
\end{lemma}
\begin{proof}
  Let us consider a sequence $f_1, \ldots, f_n, \ldots$ belonging
  to $ \B_b(\R^d, \R^+)$,
  $t \le t_1 \le \ldots \le t_n \ldots \le T$. Let $B \in \F_t$.
We first prove by induction on $n \ge 1$ that
\begin{equation}		
	\label{eq:intermediateStepMonoClassFunctions}
        \E^{s, x}[\1_B f_1(X_{t_1}) \ldots  f_n(X_{t_n})] =
\E^{s, x}\left[\1_B\E^{t, X_t}\left[f_1(X_{t_1}) \ldots f_n(X_{t_n})\right]\right].
	\end{equation}
        For $n = 1$, the property holds by Remark \ref{rmk:measurabilityMarkovExpect}
        item 2.
 Let now $n \ge 2$ and assume next that the property
\eqref{eq:intermediateStepMonoClassFunctions}
holds for $n - 1$.
By the tower property of the conditional expectation as well as
by Remark \ref{rmk:measurabilityMarkovExpect}
        item 2.
\begin{equation*}
		\begin{aligned}
                  \E^{s, x}[\1_B f_1(X_{t_1}) \ldots f_n(X_{t_n})]
  & = \E^{s, x}\left[\1_B
f_1(X_{t_1}) \ldots f_{n-1}(X_{t_{n-1}}) \E^{s, x}[f_n(X_{t_n})~|~\F_{t_{n - 1}}]\right]\\
         & = \E^{s, x}\left[\1_B
f_1(X_{t_1}) \ldots f_{n-1}(X_{t_{n-1}})
           \E^{t_{n-1},X_{t_{n-1}}}[f_n(X_{t_{n}})]\right].
	\end{aligned}
	\end{equation*}
        Now the function 
	$$
	f : (x_1, \ldots, x_{n - 1}) \in (\R^{d})^{n - 1} \mapsto f_1(x_1) \ldots f_n(x_{n - 1})
 \E^{t_{n-1}, x_{n-1}}[f_n(X_{t_n})]
        $$
	belongs to $\B_b((\R^{d})^{n - 1}, \R)$. By 
 the tower property and the induction step $n-1$
 we get \eqref{eq:intermediateStepMonoClassFunctions} for the integer $n$.

 We go on with the proof of \eqref{eq:inductionMarkov}.
 From the linearity and the monotone convergence theorem of the conditional expectation, we see that the class
 $\Lambda := \{A \in {\rm Bor}(\R^d)^{\otimes n}~|~\E^{s, x}[\1_B\1_A] = \E^{s, x}\left[\1_B\E^{t, X_t}[\1_A]\right]\}$ is a monotone class ($\lambda$-system).
 From
\eqref{eq:intermediateStepMonoClassFunctions},
 applied with $f_k = \1_{A_k}$ for some $A_k \in {\rm Bor}(\R^d)$, $1 \le k \le n$, we see that $\Lambda$ contains the $\pi$-system $\mathcal{P} := \{A \in {\rm Bor}(\R^d)^{\otimes n}~|~A = A_1 \times \ldots \times A_n, A_1, \ldots, A_n \in {\rm Bor}(\R^d)\}$.
 Since $\sigma(\mathcal{P}) = {\rm Bor}(\R^d)^{\otimes n}$ it follows from
  Theorem 3.2, Chapter 1, in \cite{billingsley}
 that
 $\Lambda  = {\rm Bor}(\R^d)^{\otimes n}.$
 Consequently for all integer $ n \ge 1$
 \begin{equation} \label{eq:induc}
	\E^{s, x}\left[\1_Bf(X_{t_1}, \ldots, X_{t_n}))\right] = \E^{s, x}\left[\1_B\E^{t, X_t}[f(X_{t_1}, \ldots, X_{t_n})]\right],
      \end{equation}
for all $f = 1_A$ and $ A \in  {\rm Bor}(\R^d)^{\otimes n}$.
Finally by approximation of any positive function $f \in \B_b((\R^d)^n, \R^+)$ by an increasing sequence of simple functions and the monotone convergence theorem for the conditional expectation,  \eqref{eq:induc} holds
for any $f \in \B_b((\R^d)^n, \R^+)$. Now
   \eqref{eq:induc} can be extended to any $f \in \B_b((\R^d)^{n}, \R)$, by setting $f = f_+ - f_-$.
This concludes the proof of \eqref{eq:inductionMarkov}.
\end{proof}
\begin{proof}
	[Proof of Proposition \ref{prop:PsxSatisfyMarkov}.] Let 
	$$
	\shh := \left\{F \in \B_b(D([t, T]), \R)~|~\E^{s, x}\left[F((X_u)_{u \in [t, T]})~|~\F_t\right] = \E^{t, X_t}[F((X_u)_{u \in [t, T]})]\right\}.
	$$
	By linearity of the conditional expectation, $\shh$ is a vector space. By monotone convergence of the conditional expectation, if $(F_n)_{n \ge 1}$ is a
        non-negative increasing  sequence of elements of $\mathcal{H}$ such that $0 \le F_n \le F_{n + 1}$ for all $n \ge 1$, then $\sup_{n \ge 1} F_n \in \mathcal{H}$. Finally let $\mathcal{C}$ be the collection
        of all
        cylindrical sets of $D([t, T], \R^d)$, i.e. the
        elements $F \in D([t, T], \R^d)$ of the form
	$$   F =  \{X_{t_1} \in A_1, \ldots, X_{t_n} \in A_n\}~|~n \in \N, t \le t_1 \le \ldots \le t_n, A_1, \ldots, A_n \in
                      {\rm Bor}(\R^d) \}.
$$
	Then we get from Lemma \ref{lemma:inductionMarkov} that $\1_C \in \mathcal{H}$ for all $C \in \mathcal{C}$ and by Theorem 21, Chapter I in \cite{DellacherieMeyer_I_IV} $\mathcal{H} = \B_b(D[t, T], \R)$.
      \end{proof}

      
      We generalize slightly Proposition \ref{prop:PsxSatisfyMarkov} in the case when $F$ is a non-negative measurable function, not necessarily bounded.
      We recall to this aim the existence of a generalized version of the conditional expectation for non-negative random variable, see Proposition
      \ref{prop:generalizedCondExp}. 

\begin{prop}
	\label{prop:markovGenCond}
	Let $(\P^{s, x})_{(s, x) \in [0, T] \times \R^d}$ be a Markov canonical class in the sense of Definition \ref{def:markovCanonicalClass}. Let $F \in \B(D([t, T], \R^d), [0, + \infty])$. Then
	\begin{equation} \label{eq:GenCond}
	\E^{s, x}\left[F((X_u)_{u \in [t, T]})\middle| \F_t\right] = \E^{t, X_t}\left[F((X_u)_{u \in [t, T]})\right]~\P^{s, x}\text{-a.s.}
	\end{equation}
\end{prop}
\begin{proof}
  Let $n \in \N$. By Proposition \ref{prop:PsxSatisfyMarkov},
\eqref{eq:GenCond} holds for $F$ replaced by $F\wedge n$, and we have
	\begin{equation}
		\label{eq:markovPropN}
		\E^{s, x}\left[F((X_u)_{u \in [t, T]})\wedge n\middle| \F_t\right] = \E^{t, X_t}\left[F((X_u)_{u \in [t, T]})\wedge n \right]~\P^{s, x}\text{-a.s.}
	\end{equation}
	On the one hand, by the  monotone convergence theorem for the conditional expectation,
$\P^{s, x}\text{-a.s.}$ we have
	\begin{equation}
		\label{eq:monotoneConvExp}
		\E^{s, x}\left[F((X_u)_{u \in [t, T]})\wedge n\middle| \F_t\right] \underset{n \rightarrow + \infty}{\longrightarrow} \E^{s, x}\left[F((X_u)_{u \in [t, T]})\middle| \F_t\right].
	\end{equation}
	On the other hand, for all $y \in \R^d$, by monotone convergence
	$$
	\E^{t, y}\left[F((X_u)_{u \in [t, T]})\wedge n\right] \underset{n \rightarrow + \infty}{\longrightarrow} \E^{t, y}\left[F((X_u)_{u \in [t, T]})\right],
	$$
	hence
	\begin{equation}
		\label{eq:monotoneConvExp2}
		\E^{t, X_t}\left[F((X_u)_{u \in [t, T]})\wedge n \right] \underset{n \rightarrow + \infty}{\longrightarrow} \E^{t, X_t}\left[F((X_u)_{u \in [t, T]})\right] ~\P^{s, x}\text{-a.s.}
	\end{equation}
	We emphasize that the conditional expectation in the right-hand side of \eqref{eq:monotoneConvExp} and \eqref{eq:monotoneConvExp2} are to be understood in the sense of Proposition \ref{prop:generalizedCondExp}.
	This shows the validity of \eqref{eq:GenCond}.
\end{proof}
\begin{prop}
	\label{lemma:markovPropGeneral}
	Let $\P \in \Pma(\Omega)$ be a Regularly Markovian probability measure in the sense of Definition \ref{def:regularlyMarkovian} with associated Markov
        canonical class $(\P^{s, x})_{(s, x) \in [0, T] \times \R^d}$ and initial law $\mu$. Then for all $t \in [0, T]$, $F \in \B_b(D([t, T], \R^d), \R)$,
	\begin{equation*}
		\E^{\P}[F((X_u)_{u \in [t, T]})|\F_t] = \E^{\P}[F((X_u)_{u \in [t, T]})|X_t]
		= \E^{t, X_t}[F((X_u)_{u \in [t, T]})] \quad \P\text{-a.s.}
	\end{equation*}
In particular $\P$ verifies the Markov Property Hypothesis \ref{hyp:markovProp}.
      \end{prop}
\begin{proof}
  We set $Z := F((X_u)_{u \in [t, T]})$. Let $A \in \F_t$. By definition of $\P$ as well as by Proposition \ref{prop:PsxSatisfyMarkov} 
we have
	\begin{equation*}
		\begin{aligned}
			\E^\P\left[Z\1_A\right] & = \int_{\R^d}\E^{0, x}[Z\1_A]\mu(dx)\\
			& = \int_{\R^d}\E^{0, x}\left[\E^{0, x}[Z|\F_t]\1_A\right]\mu(dx)\\
			& = \int_{\R^d}\E^{0, x}[\E^{t, X_t}[Z]\1_A]\mu(dx)\\
			& = \E^\P\left[\E^{t, X_t}[Z]\1_A\right]
		\end{aligned}
	\end{equation*}
	and the conclusion follows immediately from the last equality by
     the definition of the conditional expectation.
   \end{proof}
When $\P$ is Regularly Markovian, then a ''pointwise'' version of
  the intrinsic value function can be naturally provided.

  \begin{prop}
    \label{lemma:MvH2}
    For every $(s,x) \in [0,T]  \times \R^d$, the function
\begin{equation}
	\label{eq:definitionV}
	v(s,x) := \E^{s, x}\left[\exp\left(-\int_s^T f(r, X_r)dr - g(X_T)\right)\right],
      \end{equation}
is an intrinsic value function, i.e. it     is a representative of the ''class'' in $L^0$ defined in
     Definition \ref{definitionV}.
     In particular it belongs to $ \B_b([0, T] \times \R^d, \R)$ and fulfills \eqref{eq:representationV}.
 \end{prop}
\begin{proof}
  The function $v$ defined in \eqref{eq:definitionV}
	belongs to
$\B_b([0, T] \times \R^d, \R)$ by
  Remark \ref{rmk:measurabilityMarkovExpect}. 
  By the Markov property, provided by Proposition
  \ref{lemma:markovPropGeneral}, the result follows.
\end{proof}

      \medskip

      \section{Technical proofs of Section \ref{sec:extension}}

\setcounter{equation}{0}
\renewcommand\theequation{C.\arabic{equation}}
      \subsection{Proof of Lemma \ref{lemma:equivalenceProduct}}
\label{sec:proofsEquivalence}
      
	Let $\phi, \psi \in \shd(\P)$. By integration by parts, for all $t \in [0, T]$,
	\begin{equation}
		\label{eq:ippMPsiMPhi}
		\begin{aligned}
                  (\phi \psi)(t, X_t) & = (\phi \psi)(0, X_0) + \int_0^t \phi_{-}(r, X_{r})d\psi(r, X_r) +
                     \int_0^t \psi_{-}(r, X_{r}) d\phi(r, X_r) + [M[\phi], M[\psi]]_t  \\ 
			&= (\phi \psi)(0, X_0) +  \int_0^t \phi_{-}(r, X_{r})dM[\psi]_r + \int_0^t \phi_{-}(r, X_{r}) a^\P(\psi)(r, X_{r})dr \\
                                      & + \int_0^t \psi_{-}(r, X_{r})dM[\phi]_r + \int_0^t \psi_{-}(r, X_{r})a^\P(\phi)(r, X_{r}) dr + [M[\phi], M[\psi]]_t \\
                  &=  (\phi \psi)(0, X_0) + \int_0^t \phi(r, X_{r}) a^\P(\psi)(r, X_{r})dr 
                    + \int_0^t \psi(r, X_{r})a^\P(\phi)(r, X_{r}) dr \\
                  &+ \langle M[\phi], M[\psi] \rangle_t  + N_t, 
      \end{aligned}
    \end{equation}
    where $N$ is local martingale. This happens because
    $[M[\phi], M[\psi]]$ is a special semimartingale with bounded variation $\langle M[\phi], M[\psi] \rangle$.
    
We now prove the first implication $1. \Rightarrow 2.$ As $\phi \psi \in \shd(\P)$, we  have that
	\begin{equation}
		\label{eq:ippMPsiMPhiBis}
		(\phi \psi)(t, X_t)  = (\phi \psi)(0, X_0) + M_t[\phi \psi] + \int_0^t a^\P(\phi \psi)(r, X_r)dr, t \in [0,T].
       \end{equation}
       \eqref{eq:ippMPsiMPhi} and \eqref{eq:ippMPsiMPhiBis} provide
       two different decompositions of the special semimartingale
       $(\phi \psi)(t, X_t)$, which  allows to show  \eqref{eq:bracket} with \eqref{eq:gammaPsiPhiP}.
       On the other hand $(\phi \psi)(t,X_{t})$ is obviously locally square integrable
       because of \eqref{eq:ippMPsiMPhiBis} and the fact that $M[\phi \psi]$ is
       locally square integrable.
            
       We now prove the converse implication $2. \Rightarrow 1.$ By \eqref{eq:bracket}
       and \eqref{eq:ippMPsiMPhi}, taking into account \eqref{eq:gammaPsiPhiP}, we obviously get
       $$
	(\phi \psi)(t, X_t)  = (\phi \psi)(0, X_0) + N_t  + \int_0^t a^\P(\phi \psi)(r, X_r)dr, t \in [0,T],
        $$
        where $N$ is a local martingale.
       Since $(\phi \psi)(t, X_t)$ is locally square integrable then $N \in \shh^2_{loc}$ and
       so $N = M[\phi\psi)$ which shows that $\phi \psi \in \shd(\P)$.


  \subsection{Proof of Proposition \ref{prop:GammaContinuous}}
\label{prop:GammaContinuousApp}

We first recall the definition of a \textit{quasi-left continuous} process and predictable stopping time.
\begin{definition}
	\label{def:quasiLeftContinuous}
	\begin{enumerate}
		\item	A c\`adl\`ag process $X$ is called quasi-left continuous if $\Delta X_\tau = 0$ a.s. on $\{\tau < + \infty\}$ for all predictable stopping times $\tau$, see Definition 2.25 in Chapter I in \cite{JacodShiryaev}.
	\item      We recall that the notion of predictable stopping time
          is defined  for instance in Definition 2.7, Chapter I in \cite{JacodShiryaev}.
        \item We recall that given a stopping time $\tau$ the $\sigma$-field $\shf_{\tau-}$ is the $\sigma$-field
          generated by $\shf_0$ and the events of the form $A \cup \{t < \tau\}$, where $t \in \R$ and $A \in \shf_t$,
          see Definitions 1.11, Chapter I in \cite{JacodShiryaev}.
          \end{enumerate}
\end{definition}

      We continue with a lemma related to the indistinguishability of stochastic processes.
\begin{lemma}
	\label{lemma:predIndistinguable}
	Assume that $\P$ verifies Hypothesis \ref{hyp:compensator}. Let $\phi \in \shd(\P)$. Let $\Phi$ be the c\`adl\`ag modification of $\phi(\cdot, X_\cdot)$. Then $(\Phi_{t-})_{t \in [0, T]}$ and $(\phi(t, X_{t-}))_{t \in [0, T]}$ are $\P$-indistinguishable.
\end{lemma}
\begin{proof}

Let $\tau$ be a predictable stopping time.  
  By Theorem 86, Chapter IV in \cite{DellacherieMeyer_I_IV}, it will be enough to prove
	\begin{equation}
		\label{eq:equalityPredictableTime2}
		\Phi_{\tau-} = \phi(\tau, X_{\tau-}) \quad \P\text{-a.s.}
              \end{equation}
              on $\{\tau < +\infty \}$. 
We write
	\begin{equation} \label{eq:stopping}
          \Phi_t = M[\phi]_{t} +
     \Phi_0  +   \int_0^{t} a^\P(\phi)(r, X_r)dr, \ t \in [0,T].
	\end{equation}
        Let now $(\shs_n)_{n \ge 1}$ be a localizing sequence for $M[\phi]$ verifying $\int_0^{T \wedge \shs_n} |a^\P(\phi)(r, X_r)|dr \le n$.
        It will be sufficient to prove
        	\begin{equation}
		\label{eq:equalityPredictableTime}
		\Phi_{\tau-}\1_{\{\tau \le \shs_n\}}
                = \phi(\tau, X_{\tau-})\1_{\{\tau \le \shs_n\}} \quad \P\text{-a.s.}
	\end{equation}
        To prove what precedes,
     on the one hand, taking the limit $t \rightarrow (\tau \wedge \shs_n)-$ 
	 in \eqref{eq:stopping} yields, 
	\begin{equation}
		\label{eq:predIndis2}
		\Phi_{(\tau \wedge \shs_n)-} = M[\phi]_{(\tau \wedge \shs_n)-} + \Phi_0
                + \int_0^{\tau \wedge \shs_n} a^\P(\phi)(r, X_r)dr.
	\end{equation}
	On the other hand,
        by 1.17, Chapter I in \cite{JacodShiryaev},
        with $A = \Omega, S = \shs_n, T = \tau,$ for all $n \ge 1$, we get
$\{\shs_n < \tau\}  \in  \shf_{\tau-},$ and therefore
      $$ \{\tau \le \shs_n\} \in \shf_{\tau-}, n \ge 1.$$
        Setting $X = M[\phi]^{\shs_n}$ in  Lemma 2.27, Chapter I in \cite{JacodShiryaev}, we have
        that
\begin{equation} \label{eq:MPhiTau}
  M[\phi]_{(\tau \wedge \shs_n)-} = \E^\P\left[M[\phi]_{\tau \wedge \shs_n} | \shf_{\tau-}\right] \ \text{on} \ \{\tau \le \shs_n\}.
\end{equation}
As $\int_0^{T \wedge \shs_n} |a^\P(\phi)(r, X_r)|dr \le n  \in L^1(\P)$,
by \eqref{eq:MPhiTau} and \eqref{eq:stopping} evaluated at $t = \tau \wedge \shs_n$  we get
	\begin{eqnarray*}
 M[\phi]_{(\tau \wedge \shs_n)-} &=& \E^\P\left[M[\phi]_{\tau \wedge \shs_n} | \shf_{\tau-}\right]   \\
  &=&      \E^\P\left[\Phi_{\tau \wedge \shs_n} | \shf_{\tau-}\right] - \phi(0, X_0) - \int_0^{\tau \wedge \shs_n}a^\P(\phi)(r, X_r)dr \quad \text{on} \quad \{\tau \le \shs_n\}.
	\end{eqnarray*}
	Replacing  $M[\phi]_{(\tau \wedge \shs_n)-}$  in \eqref{eq:predIndis2} we get
	\begin{equation*}
		\Phi_{(\tau \wedge \shs_n)-} = \E^\P\left[\Phi_{\tau \wedge \shs_n}  | \shf_{\tau-}\right] \quad \text{on} \quad \{\tau \le \shs_n\},
	\end{equation*}    
	that is
	\begin{equation}
		\label{eq:predIndis2bis}
		\Phi_{\tau-}\1_{\{\tau \le \shs_n\}} = \E^\P\left[\phi(\tau, X_\tau) | \shf_{\tau-}\right]\1_{\{\tau \le \shs_n\}} \quad \P\text{-a.s.}
	\end{equation}
        Now 	Hypothesis \ref{hyp:compensator} implies that $\nu^{X, \P}(X, \{t\} \times \R^d) = 0$ identically, so by   Corollary 1.19, Chapter II in \cite{JacodShiryaev}, the process $X$ is quasi-left continuous under $\P$ in the sense of Definition \ref{def:quasiLeftContinuous} and we have $\Delta X_\tau = 0$ 
on $\{\tau < +\infty\}$ $\P$-a.s.
 Hence $\phi(\tau, X_\tau) = \phi(\tau, X_{\tau-})$ $\P$-a.s. Moreover, since $\tau$ is $\shf_{\tau-}$-measurable by 1.14, Chapter I of \cite{JacodShiryaev}, then $\phi(\tau, X_{\tau-})$ is $\shf_{\tau-}$-measurable. Consequently  \eqref{eq:predIndis2bis} then yields
	\begin{equation}
		\label{eq:equalityPredictableTimeBis}
		\Phi_{\tau-}\1_{\{\tau \le \shs_n\}} = \phi(\tau, X_{\tau-}) \1_{\{\tau \le \shs_n\}}  \quad \P\text{-a.s.}
	\end{equation}
and therefore
\eqref{eq:equalityPredictableTime}.
This concludes the proof.
\end{proof}

\medskip

We continue with the proof of the aforementioned proposition.

\begin{proof}[Proof of Proposition \ref{prop:GammaContinuous}]
  Let $\phi \in \shd$.  Using the notation \eqref{eq:defMPhiP}, let
  $M[\phi]$ and $M[v]$ be the càdlàg local martingales, which belong to $\shh_{loc}^2(\P)$, also taking into account Proposition \ref{prop:MarkovMart}.
  Hence $[M[v], M[\phi]] \in \sha_{loc}(\P)$ by Proposition 4.51, Chapter I in \cite{JacodShiryaev}, which, taking into account
  \begin{equation} \label{eq:vpĥi}
  [M[v], M[\phi]] = [M[v]^c, M[\phi]^c] + [M[v]^d, M[\phi]^d],
\end{equation}
 yields $[M[v]^d, M[\phi]^d] \in \sha_{loc}(\P)$.
  
	By Theorem 4.52, Chapter I in \cite{JacodShiryaev}, we have
	\begin{equation}
		\label{eq:bracketMPhid}
		[M[v]^d, M[\phi]^d] = \sum_{0 < r \le \cdot}\Delta M[v]_r\Delta M[\phi]_r = \sum_{0 < r \le \cdot}\Delta V_r \Delta \Phi_r,
	\end{equation}
	where $V$ and $\Phi$ are c\`adl\`ag versions of $v(\cdot, X_\cdot)$ and $\phi(\cdot, X_\cdot)$ respectively. By Lemma \ref{lemma:predIndistinguable}, $\P$-a.s., for all $r \in [0, T]$, 
	$$
	\Delta V_r = V_r - V_{r-} = v(r, X_r) - v(r, X_{r-}) = v(r, X_{r-} + \Delta X_r) - v(r, X_{r-})
	$$
	and
	$$
	\Delta \Phi_r = \Phi_r - \Phi_{r-} = \phi(r, X_r) - \phi(r, X_{r-}) = \phi(r, X_{r-} + \Delta X_r) - \phi(r, X_{r-}).
	$$
                	Therefore equality \eqref{eq:bracketMPhid}
        gives
	\begin{eqnarray}
		\label{eq:bracketMPhid2}
		[M[v]^d, M[\phi]^d] &=& \sum_{0 < r \le \cdot}(v(r, X_{r-} + \Delta X_r) - v(r, X_{r-}))(\phi(r, X_{r-} + \Delta X_r) - \phi(r, X_{r-}))\\
                                     &=&      \int_{]0, \cdot] \times \R^d} W(r, X_{r-}, q)
                                         \mu^X(dr, dq),    
                                \nonumber
        \end{eqnarray}
        where the equality holds up to indistinguishability and $W$ was defined in \eqref{eq:martcomTer}.
	Since $[M[v]^d, M[\phi]^d] \in \sha_{loc}(\P)$, 
        $ \int_{]0, \cdot] \times \R^d}W(r, X_{r-}, q)\nu^{X, \P}(dr, dq) $ is well-defined and it also belongs to $\sha_{loc}(\P)$.
           This establishes the first item of the proposition.

           Moreover
	\begin{equation} \label{eq:martcomBis}
		\int_{]0, \cdot] \times \R^d}W(r, X_{r-}, q)\mu^X(dr, dq) - \int_{]0, \cdot] \times \R^d}W(r, X_{r-}, q)\nu^{X, \P}(dr, dq)
	\end{equation}
	is a local martingale. 
        Consequently  $[M[v]^d, M[\phi]^d] - \int_{]0, \cdot] \times \R^d}W(r, X_{r-}, q)\nu^{X, \P}(dr, dq)$ is a local martingale.
        Since $[M[v]^d, M[\phi]^d] - \int_{]0, \cdot] \times \R^d}W(r, X_{r-}, q)\nu^{X, \P}(dr, dq)$ is a local martingale
        and the process  $\int_{]0, \cdot] \times \R^d}W(r, X_{r-}, q)\nu^{X, \P}(dr, dq)$ is predictable, we have that
	\begin{equation}
		\label{eq:compensatorDiscontinuous}
		\langle M[v]^d, M[\phi]^d\rangle = \int_{]0, \cdot] \times \R^d}W(r, X_{r-}, q)\nu^{X, \P}(dr, dq) \quad \P\text{-a.s.}
              \end{equation}
Now \eqref{eq:vpĥi}
        implies
        $\langle M[v], M[\phi] \rangle = [M[v]^c, M[\phi]^c] + \langle M[v]^d, M[\phi]^d \rangle$.
        By \eqref{eq:bracketGamma}
        and \eqref{eq:compensatorDiscontinuous}, we have that
	$$
	\int_0^\cdot \Gamma^v(\phi)(r, X_r)dr = [ M[v]^c, M[\phi]^c ] + \int_{]0, \cdot] \times \R^d}W(r, X_{r-}, q)\nu^{X, \P}(dr, dq) \quad \P\text{-a.s.},
	$$
	which immediately yields
	$$
	\int_0^\cdot \Gamma^{v, c}(\phi)(r, X_r)dr = [ M[v]^c, M[\phi]^c ] \quad \P\text{-a.s.},
	$$
        where $\Gamma^{v,c}$ is the linear operator defined in \eqref{eq:decompGamma}.
	This concludes the proof of the second item.
\end{proof}

\subsection{Proof of Proposition \ref{prop:closureWeakDirichlet}}

\label{prop:closureWeakDirichletApp}

  Let $(\phi_n)_{n \ge 1}$ be a sequence of elements of $\shd$ such that $d_\scrd(\phi_n, \phi) \underset{n \rightarrow + \infty}{\longrightarrow} 0$. In particular, $d_c(\phi_n, \phi_m) \underset{n, m \rightarrow + \infty}{\longrightarrow} 0$, that is $[ M[\phi_n]^c - M[\phi_m]^c ]_T \overset{\P}{\underset{n, m \rightarrow + \infty}{\longrightarrow}} 0$.
We consider the unique special semimartingale decomposition 
	$$
	\phi_n(\cdot, X_\cdot) = M[\phi_n]^c + M[\phi_n]^d + \int_0^\cdot a^\P(\phi_n)(r, X_r)dr =: M[\phi_n]^c + A(\phi_n),
	$$
        where $M[\phi_n]^c$ (resp. $M[\phi_n]^d$) is a continuous (resp. purely discontinuous) local martingale
        and $A(\phi_n)$ are the sum of a continuous finite
        variation process and a
        purely discontinuous local martingale, therefore
        martingale orthogonal.
        By Problem 5.25, Chapter 1 in \cite{karatshreve},
        the sequence $(M[\phi_n]^c)_{n \ge 1}$ is a Cauchy sequence in $\D^{ucp}$. Consequently there exists a continuous process $M$ such that $M[\phi_n]^c
        \underset{n \rightarrow + \infty}{\longrightarrow} M$ $\P$-u.c.p. Since the space of continuous $\P$-local martingales vanishing at $0$ is closed under u.c.p. convergence, $M$ is a continuous $\P$-local martingale.
Clearly $\phi_n(\cdot, X) \rightarrow \phi(\cdot, X_\cdot)$ $dt \otimes d\P$-a.e.

We set $A(\phi) := \phi(\cdot, X_\cdot) - M$.
Let $N$ be a continuous $\P$-local martingale.
We first observe that $[\phi(\cdot, X_\cdot),N]$ exists by item 3. of Definition \ref{def:closureDomain}
so that  $[A(\phi), N]$ exists. It remains to prove 
\begin{equation} \label{eq:APhiN}
  [A(\phi), N] = 0,
\end{equation}
so that
$\phi(\cdot, X_\cdot)$ is a weak Dirichlet process with decomposition $ M + A(\phi)$.
   For this purpose, we first observe that \eqref{eq:APhiN} holds for $\phi$ replaced by $\phi_n \in \shd$ since
    $A(\phi_n)$ are martingale orthogonal.
    \eqref{eq:APhiN} is then a consequence of the continuity of the map
$\psi \mapsto [A(\psi),N]$,  from $\scrd$ to $\D^{ucp}$.
Indeed this holds because  item 1. of Remark \ref{rmk:equivalence}
implies that $\phi \mapsto [\phi(\cdot, X_\cdot),N]$ is continuous
and $\phi \mapsto [M[\phi]^c, N]$ is also continuous by Kunita-Watanabe inequality,
    taking into account that $d_c < d_\scrd$.

  \subsection{Proof of Proposition \ref{prop:extensionGammaClosure}}

\label{prop:extensionGammaClosureApp}
  
  Since $d_\scrd$ is homogeneous, in order to prove the continuity extension property,
  it is enough to check the continuity of $\Gamma^{v, c}$ in $0$. Let then $(\phi_n)_{n \ge 1}$ be a sequence of elements of $\shd$ such that $d_\scrd(\phi_n, 0) \underset{n \rightarrow + \infty}{\longrightarrow} 0$. Previous convergence implies that $d_c(\phi_n, 0) \underset{n \rightarrow + \infty}{\longrightarrow} 0$, hence $[M[\phi_n]^c]_T \overset{\P}{\underset{n \rightarrow + \infty}{\longrightarrow}} 0$ and up to a subsequence we can assume that

	\begin{equation}
		\label{eq:extensionGammaConvAs}
		[M[\phi_n]^c]_T \underset{n \rightarrow + \infty}{\longrightarrow} 0 \quad \P\text{-a.s.}
	\end{equation}
	By \eqref{eq:defGammaContinuous} in Proposition \ref{prop:GammaContinuous},
	\begin{equation}
		\label{eq:gammaVPhin}
		\int_0^\cdot\Gamma^{v, c}(\phi_n)(r, X_r)dr = [ M[v]^c, M[\phi_n]^c].
	\end{equation}
        By Kunita-Watanabe inequality, for all $0 \le s < t$,
	\begin{equation}
		\label{eq:kunitaWatanabe}
			|d[M[v]^c, M[\phi_n]^c]|(]s, t])  \le \sqrt{[M[v]^c]_t - [M[v]^c]_s }\sqrt{[M[\phi_n]^c]_t - [M[\phi_n]^c]_s},
	\end{equation}
	hence, by Cauchy-Schwarz inequality,
	\begin{equation}
		\label{eq:convTotalVariation}
		\begin{aligned}
			& \sup \left\{\sum_{i = 1}^p |d[M[v]^c, M[\phi_n]^c]|(]t_{i - 1}, t_i])\right\}\\
			\le &  \sqrt{\left \{\sum_{i = 1}^p ([M[v]^c]_{t_{i+1}} - [M[v]^c]_{t_i}) \right\}} \sqrt{\left\{\sum_{i = 1}^p ([M[\phi_n]^c]_{t_{i+1}} - [M[\phi_n]^c]_{t_i})\right \}}\\
			= & \sqrt{[M[v]^c]_T}\sqrt{[ M[\phi_n]^c]_T},
		\end{aligned}
	\end{equation}
	where the supremum is taken over all subdivisions $0 = t_0 < t_1 < \ldots < t_p = T$ of $[0, T]$. Inequality \eqref{eq:convTotalVariation} and convergence \eqref{eq:extensionGammaConvAs} then imply that $d[ M[v]^c, M[\phi_n]^c] \rightarrow 0$ in the total variation norm for signed measures on $[0, T]$. This immediately yields by \eqref{eq:gammaVPhin} that
	$$
	\int_0^T |\Gamma^{v, c}(\phi_n)|(r, X_r) dr \underset{n \rightarrow + \infty}{\longrightarrow} 0 \quad \P\text{-a.s.},
	$$
	hence $d_{L^0}(\phi_n, 0) \underset{n \rightarrow + \infty}{\longrightarrow} 0$ so that  $\Gamma^{v, c}$
       indeed extends continuously to $\scrd$

        Concerning \eqref{eq:defGammaContinuousExt}, let $\phi \in \scrd$ and
        $(\phi_n)$ converging to $\phi$ in $d_\scrd$.
For $t \in [0,T]$
        we write 
        \begin{eqnarray} \label{eq:SubsIneq}
          \left\vert  [ M[v]^c, M[\phi]^c]_t - \int_0^t\Gamma^{v, c}(\phi)(r, X_r)dr \right\vert 
          &\le & \vert [ M[v]^c, M[\phi - \phi_n]^c]_t\vert +   \left\vert [ M[v]^c, M[\phi_n]^c]_t -   \int_0^t\Gamma^{v, c}(\phi_n)(r, X_r)dr\right\vert \nonumber \\
&+&    \int_0^t \vert \Gamma^{v, c}(\phi_n-\phi)\vert(r, X_r)dr.
\end{eqnarray}
Using Yamada-Watanabe inequality, the fact that $d_c + d_{L^0} < d_\scrd$,
taking into account \eqref{eq:defGammaContinuous} in Proposition \ref{prop:GammaContinuous}
and the continuity of $\Gamma^{v, c}$, 
we can take the limit on the right-hand side of \eqref{eq:SubsIneq},
which converges to  zero in probability.
       This concludes the proof of \eqref{eq:defGammaContinuousExt}.

\section{An integrability property of the intrinsic value
  function $v$
  defined in Definition \ref{definitionV}}

\setcounter{equation}{0}
\renewcommand\theequation{D.\arabic{equation}}
In this short section we state and prove a result regarding an integrability property of the functional $v(\cdot, \cdot + q)/v$ under $\Q$ with respect to the compensator $\nu^{X, \P}(X, dt, dq) = dtL(t, X_{t-}, dq)$.

\begin{lemma}
  \label{lemma:intCompensator}
  Let $k$ be a truncation function.
   Assume Hypothesis \ref{hyp:compensator}.
Then the function defined in \eqref{eq:intfv} belongs to
        $L^0(\P)$.

      \end{lemma}

\begin{proof}
We set $Y(t, X_{t-}, q) := \frac{v(t, X_{t-} + q)}{v(t, X_{t-})}$.
It is enough to prove
\begin{equation}\label{eq:finite-nu}
  \left(k\left\vert Y - 1 \right\vert\right)*\nu_T^{X, \P} < +\infty \  \P{\rm -a.s.}
  \end{equation}
  We will make use of the density process $D$
between $\Q $ and $\P$,
which was defined in Notation \ref{not:MartD}.
Without restriction of generality we can suppose here $k$ to be non-negative.
By definition of the compensator $\nu^{X, \P}$ of $\mu^X$ under $\P$, see Theorem 1.8, Chapter I,
in \cite{JacodShiryaev}.
Let $\tau$ be a fixed stopping time. We have
\begin{equation}
		\label{eq:intComp1}
		\begin{aligned}
                  \E^{\P}\left[\left(k D_{-}\left\vert Y - 1\right\vert\right)*\nu^{X, \P}_\tau\right] &  =
                  \E^{\P}\left[ \left(k D_{-}\left\vert Y - 1\right\vert\right)*\mu^X_\tau\right]\\
& = \E^{\P}\left[\sum_{0 < r \le \tau} D_{r-} k(\Delta X_r)\left\vert \frac{v(r, X_{r-} + \Delta X_r)}{v(r, X_{r-})} - 1 \right\vert\right].
		\end{aligned}
	\end{equation}
	Recall that $v \in \shd(\P)$ by Proposition \ref{prop:MarkovMart} item 2.,
keeping in mind Definition \ref{def:domainP}.
Then by Lemma \ref{lemma:predIndistinguable}, the processes $(v(t, X_{t-}))_{t \in [0, T]}$ and $(V_{t-})_{t \in [0, T]}$ are $\P$-indistinguishable, where $V$ is given by \eqref{eq:VD} in Notation \ref{not:MartD}, also taking into account
 Proposition \ref{prop:MarkovMart} item 1.
 Setting $C :=  1/V_0$, \eqref{eq:VD} yields
 the following two statements.
	\begin{enumerate}
        \item The processes $(D_{t-})_{t \in [0, T]}$ and
$\left(C\exp\left(-U^0_t\right)V_{t-})\right)_{t \in [0, T]}$
are $\P$-indistinguishable, where $U^0$ is defined by \eqref{eq:definitionU}.
	        \item          We have
                  $$ \Delta D_t  = \left(C\exp(-U^0_t)(V_t - V_{t-}) = C\exp(-U^0_t)(v(t, X_{t-} + \Delta X_t) - v(t, X_{t-}))\right),$$
   $ t \in [0,T]$,   where previous equalities hold in the sense of $\P$-indistinguishability.

	\end{enumerate} 
	Equality \eqref{eq:intComp1} then rewrites
	\begin{equation}
		\label{eq:intComp2}
		\begin{aligned}
	\E^{\P}\left[\left(D_{-}k\left\vert Y - 1\right\vert\right)*\nu^{X, \P}_\tau\right] & = \E^{\P}\left[\sum_{0 < r \le \tau} C\exp(-U^0_{r})v(r, X_{r-})k(\Delta X_r)\left\vert \frac{v(r, X_{r-} + \Delta X_r)}{v(r, X_{r-})} - 1 \right\vert\right]\\
			& = \E^{\P}\left[\sum_{0 < r \le \tau}k(\Delta X_r)\left\vert C\exp(-U^0_{r})\left(v(r, X_{r-} + \Delta X_r) - v(r, X_{r-})\right)\right\vert\right]\\
			& = \E^\P\left[\sum_{0 < r \le \tau}k(\Delta X_r)|\Delta D_r|\right],
		\end{aligned}
	\end{equation} 
	where we have used item 1. above for the first equality, and item 2. above for the third equality. We now consider the local martingale $M := k*(\mu^X - \nu^{X, \P})$, which is well-defined since $k \in \shg_{\rm loc}(\nu^{X, \P})$
        (that notation is borrowed from Definition 1.27, Chapter II
        of \cite{JacodShiryaev}) 
        taking into account Hypothesis \ref{hyp:compensator}.
	By definition the jumps of the local martingale $M$ are bounded by $\Vert k\Vert_\infty$. Hence by Lemma 2.6  in \cite{BandiniRusso_equivalence} and  Lemma 3.14, Chapter III
        in \cite{JacodShiryaev}, the process
        $[M, D] = \sum_{t \le \cdot} (\Delta M)_t (\Delta D)_t   = \sum_{t \le \cdot} k(\Delta X_t) \Delta D_t $ belongs to
        $\sha_{loc}(\P)$, which yields $\left(D_{-}k\left\vert Y - 1\right\vert\right)*\nu^{X, \P} \in \sha_{loc}(\P)$ by \eqref{eq:intComp2}. In particular, this implies that
        \begin{equation}
        	\label{eq:boundP}
        	\left(D_{-}k\left\vert Y - 1\right\vert\right)*\nu^{X, \P}_T = \int_0^TD_{t-}\left(\int_{\R^d}k(q)|Y(t, X_{t-}, q) - 1|L(t, X_{t-}, dq)\right)dt < + \infty,~\P\text{-a.s.}
        \end{equation}
        Since $\Q \ll \P$, previous inequality also holds $\Q$-a.s.
         By Proposition 3.5, Chapter III in \cite{JacodShiryaev} applied with $P = \P$ and $P' = \Q$, $D > 0$ $\Q$-a.s.
        we know that $\inf_{t \in [0,T]} D_{t-} > 0$, $\Q$-a.s.
        and therefore $\P$-a.s. since $\Q$ and $\P$ are equivalent.
       Consequently
        \begin{equation*}
        	\begin{aligned}
                  \left(k\left\vert Y - 1\right\vert\right)*\nu^{X, \P}_{T} &=
 \int_0^{T} \left(\int_{\R^d}k(q)|Y(t, X_{t-}, q) - 1|L(t, X_{t-}, dq)\right)dt\\
        		& = \int_0^{T} \frac{1}{D_{t-}}D_{t-}\left(\int_{\R^d}k(q)|Y(t, X_{t-}, q) - 1|L(t, X_{t-}, dq)\right)dt\\
    & \le \frac{1}{\inf_{t \in [0,T]} D_{t-}}  \int_0^TD_{t-}\left(\int_{\R^d}k(q)|Y(t, X_{t-}, q) - 1|L(t, X_{t-}, dq)\right)dt < +\infty, \ \P-{\rm a.s.}
        	\end{aligned}
        \end{equation*}
because of \eqref{eq:boundP}. This concludes the proof of \eqref{eq:finite-nu}.
\end{proof}

\section{Extension to mean-field optimization}
\setcounter{equation}{0}
\renewcommand\theequation{E.\arabic{equation}}
This short section is devoted to the proof of the equivalence between the optimization problems \eqref{eq:generalizedCharacterization} and \eqref{eq:linearizedProblem}.
\begin{lemma}
	\label{lemma:equivalenceOpti}
	Let $F : \R \rightarrow \R$ be a convex differentiable function. Let $\tilde \Q \in \shp(\Omega)$ be the solution to Problem \eqref{eq:generalizedCharacterization}. Then $\tilde \Q$ is solution to the linearized Problem \eqref{eq:linearizedProblem} with  $\tilde\vphi(X) := F'\left(\E^{\tilde \Q}[\vphi(X)]\right)\vphi(X)$.
\end{lemma}
\begin{proof}

  Let $\lambda \in [0, 1]$. By definition, for all $\Q \in \shp(\Omega)$,
	$$
	F\left(\E^{\lambda \Q + (1 - \lambda)\tilde \Q}[\vphi(X)]\right) + H(\lambda \Q + (1 - \lambda)\tilde \Q | \P) - F\left(\E^{\tilde \Q}[\vphi(X)]\right) - H(\tilde \Q | \P) \ge 0,
	$$
	that is
	\begin{equation}
		\label{eq:linInter1}
		F\left(\lambda \E^\Q[\vphi(X)] + (1 - \lambda)\E^{\tilde \Q}[\vphi(X)]\right) - F\left(\E^{\tilde \Q}[\vphi(X)]\right) + H(\lambda \Q + (1 - \lambda)\tilde \Q | \P) - H(\tilde \Q | \P) \ge 0.
	\end{equation}
	By the convexity of the relative entropy, see Remark \ref{rmk:relativeEntropy} item 1., we have
	\begin{equation}
		\label{eq:convexH}
		H(\lambda \Q + (1 - \lambda)\tilde \Q | \P) - H(\tilde \Q | \P) \le \lambda(H(\Q | \P) - H(\tilde \Q | \P)).
	\end{equation}
	Combining \eqref{eq:linInter1} and \eqref{eq:convexH} and dividing by $\lambda$ we get
	\begin{equation*}
		\frac{1}{\lambda}\left(F\left(\lambda \E^\Q[\vphi(X)] + (1 - \lambda)\E^{\tilde \Q}[\vphi(X)]\right) - F\left(\E^{\tilde \Q}[\vphi(X)]\right) \right) + H(\Q | \P) - H(\tilde \Q | \P) \ge 0,
	\end{equation*}
	and letting $\lambda \rightarrow 0$ yields
	\begin{equation*}
		F'\left(\E^{\tilde \Q}[\vphi(X)]\right)(\E^\Q[\vphi(X)] - \E^{\tilde \Q}[\vphi(X)]) + H(\Q | \P) - H(\tilde \Q | \P) \ge 0,
	\end{equation*}
	which rewrites
	\begin{equation*}
		F'\left(\E^{\tilde \Q}[\vphi(X)]\right)\E^\Q[\vphi(X)] + H(\Q | \P) \ge F'\left(\E^{\tilde \Q}[\vphi(X)]\right)\E^{\tilde \Q}[\vphi(X)] + H(\tilde \Q | \P).
	\end{equation*}
	We conclude from the previous inequality that $\tilde \Q$ is a solution of Problem \eqref{eq:linearizedProblem}.
\end{proof}

\section*{Acknowledgments}

The research of the first named author was supported by a doctoral fellowship
PRPhD 2021 of the Région Île-de-France.
The research of the second and third named authors was partially
supported by the  ANR-22-CE40-0015-01 project (SDAIM).
The authors are extremely grateful to the anonymous Referee for the careful reading
which has stimulated them to considerably improve the first submitted
version.

\bibliographystyle{plain}
\bibliography{ThesisBourdaisFinal
}

\end{document}